\documentclass[times]{speauth}

\usepackage{moreverb}

\usepackage{amsfonts,amssymb,amsthm,newlfont,graphicx,subfigure,algorithm,fixltx2e,booktabs,bm}
\usepackage[noend]{algpseudocode}
\usepackage[T1]{fontenc}
\usepackage{epstopdf}
\usepackage{empheq} 
\usepackage[colorlinks,bookmarksopen,bookmarksnumbered,citecolor=red,urlcolor=red]{hyperref}
\numberwithin{algorithm}{section}

\newcommand\BibTeX{{\rmfamily B\kern-.05em \textsc{i\kern-.025em b}\kern-.08em
T\kern-.1667em\lower.7ex\hbox{E}\kern-.125emX}}

\def\diag{\mathop{\mathrm{diag}}}
\newtheorem{thm}{Theorem}[section]
\newtheorem{lem}{Lemma}[section]

\newtheorem{cor}{Corollary}[section]

\numberwithin{equation}{section}
\renewcommand{\theequation}{\thesection.\arabic{equation}}
\def\simgt{\,\hbox{\lower0.6ex\hbox{$>$}\llap{\raise0.3ex\hbox{$\sim$}}}\,}
\def\simlt{\,\hbox{\lower0.6ex\hbox{$<$}\llap{\raise0.3ex\hbox{$\sim$}}}\,}
\def\simgteq{\,\hbox{\lower0.6ex\hbox{$\ge$}\llap{\raise0.6ex\hbox{$\sim$}}}\,}
\def\simlteq{\,\hbox{\lower0.6ex\hbox{$\le$}\llap{\raise0.6ex\hbox{$\sim$}}}\,}
\makeatletter
\def\user@resume{resume}
\def\user@intermezzo{intermezzo}
\newcounter{previousequation}
\newcounter{lastsubequation}
\newcounter{savedparentequation}
\renewenvironment{subequations}[1][]{%
      \def\user@decides{#1}%
      \setcounter{previousequation}{\value{equation}}%
      \ifx\user@decides\user@resume
           \setcounter{equation}{\value{savedparentequation}}%
      \else
      \ifx\user@decides\user@intermezzo
           \refstepcounter{equation}%
      \else
           \setcounter{lastsubequation}{0}%
           \refstepcounter{equation}%
      \fi\fi
      \protected@edef\theHparentequation{%
          \@ifundefined {theHequation}\theequation \theHequation}%
      \protected@edef\theparentequation{\theequation}%
      \setcounter{parentequation}{\value{equation}}%
      \ifx\user@decides\user@resume
           \setcounter{equation}{\value{lastsubequation}}%
         \else
           \setcounter{equation}{0}%
      \fi
      \def\theequation  {\theparentequation  \alph{equation}}%
      \def\theHequation {\theHparentequation \alph{equation}}%
      \ignorespaces
}{%
  \ifx\user@decides\user@resume
       \setcounter{lastsubequation}{\value{equation}}%
       \setcounter{equation}{\value{previousequation}}%
  \else
  \ifx\user@decides\user@intermezzo
       \setcounter{equation}{\value{parentequation}}%
  \else
       \setcounter{lastsubequation}{\value{equation}}%
       \setcounter{savedparentequation}{\value{parentequation}}%
       \setcounter{equation}{\value{parentequation}}%
  \fi\fi
  \ignorespacesafterend
}
\makeatother

\begin{document}

\runningheads{Kareem T. Elgindy and Hareth M. Refat}{Numerical Solution of Lane-Emden Differential Equations}

\title{High-Order Shifted Gegenbauer Integral Pseudospectral Method for Solving Differential Equations of Lane-Emden Type}

\author{Kareem T. Elgindy\affil{1},
Hareth M. Refat\affil{2}\corrauth}
\address{\affilnum{1}Mathematics Department, Faculty of Science, Assiut University, Assiut 71516, Egypt
\break
\affilnum{2}Mathematics Department, Faculty of Science, Sohag University, Sohag 82524, Egypt}
\corraddr{Mathematics Department, Faculty of Science, Sohag University, Sohag 82524, Egypt. E-mail: harith\char`_refaat@science.sohag.edu.eg; hareth.mohamed.refat@gmail.com.}

\begin{abstract}
We present a novel, high-order, efficient, and exponentially convergent shifted Gegenbauer integral pseudospectral method (SGIPSM) to solve numerically Lane-Emden equations provided with some mixed Neumann and Robin boundary conditions. The framework of the proposed method includes: (i) recasting the problem into its integral formulation, (ii) collocating the latter at the shifted flipped-Gegenbauer-Gauss-Radau (SFGGR) points, and (iii) replacing the integrals with accurate and well-conditioned numerical quadratures constructed via SFGGR-based shifted Gegenbauer integration matrices. The integral formulation is eventually discretized into linear/nonlinear system of equations that can be solved easily using standard direct system solvers. The implementation of the proposed method is further illustrated through four efficient computational algorithms. Moreover, we furnish rigorous error and convergence analyses of the SGIPSM. Five numerical test examples are presented to verify the effectiveness, accuracy, exponential convergence, and numerical stability of the proposed method. The numerical simulations are associated with extensive numerical comparisons with other rival methods in the literature to demonstrate further the power of the proposed method. The SGIPSM is broadly applicable and represents a strong addition to common numerical methods for solving linear/nonlinear differential equations when high-order approximations are required using a relatively small number of collocation points. 
\end{abstract}


\maketitle

\vspace{-6pt}

\section{Introduction}
\label{Int}
Lane-Emden equations describe various physical phenomena in mathematical physics and astrophysics such as thermal explosions, stellar structure, thermal behavior of a spherical cloud of gas, thermionic currents, radiative cooling, etc.; cf. \cite{chandrasekhar1958introduction,Harley2007,Singh2009,kanth2010he,yuzbacsi2013improved,tohidi2013new,abd2015new}. They arise also in chemistry, chemical engineering, and other branches to govern the concentrations of oxygen and the carbon substrate \cite{Wazwaz2016}, deriving analytical expressions for effectiveness factors $\eta$ in non-isothermal, spherical catalysts \cite{Scott1984}, etc.

In the present work, we are concerned with the numerical solution of the following two forms of nonlinear and linear Lane-Emden equations
\begin{equation}\label{eq:1}
{y''}(x) + \frac{\alpha_{2}}{x} {y'}(x) + f(x,y) = 0,\quad x\in (0,b],
\end{equation}
\begin{equation}\label{eq:4}
{y''}(x) + \frac{\alpha_{2}}{x} {y'}(x) + p(x)\,y(x) = g(x),\quad x \in (0,b],
\end{equation}
respectively, where $y$ is the unknown solution function, $f$ is a given nonlinear smooth function, $p(x)$ and $g(x)$ are some given sufficiently smooth functions, $\alpha_{2} \in \mathbb{R}$ and $b \in \mathbb{R}^+$ are some given constants. Both equations are provided with the following mixed Neumann and Robin boundary conditions 
\begin{subequations}
\begin{equation}\label{eq:2}
{y'}(0) = \alpha_{1},
\end{equation}
\begin{equation}\label{eq:3}
\beta y(b) + \gamma {y'}(b)=\delta,
\end{equation}
\end{subequations}
for some given real numbers $\alpha_{1}, \beta, \gamma$, and $\delta$. The real parameters $\beta$ and $\gamma$ in the boundary condition \eqref{eq:3} do not vanish simultaneously. We shall refer to Eqs. \eqref{eq:1}, \eqref{eq:2}, and \eqref{eq:3} by Problem 1. Also, we refer to Eqs. \eqref{eq:4}, \eqref{eq:2}, and \eqref{eq:3} by Problem 2. 

The theoretical picture for second-order initial-value problems of Lane-Emden type equations is well-known. Local existence, global existence, and uniqueness of solutions were proven by \cite{Biles2002}. For problem 1 with $f(x,y)=y^{m}$ and $\alpha_{2}=2$, the exact solution is available only for $m=0,1$, and $5$. For other values of $m$, the problem can only be solved numerically \cite{Mohan1980,Parks1984power}; therefore, the problem of finding solutions for other values of $m$ is very important from both theoretical and practical viewpoints. This topic captivated the interest of many researchers through many decades in attempt to derive new efficient computational algorithms to find accurate numerical solutions for these types of singular differential equations. In the following, we mention a few: \cite{Seidov1977solution,Mohan1980,Parks1984power} approached the numerical solution using power series methods. \cite{Horedt1986} published an article later listing seven-digit tables of the numerical solutions of the Lane-Emden equation for the plane-parallel, cylindrical, and spherical case using a Runge-Kutta method-- a work considered a succession to the works of \cite{Sadler1932} and \cite{Chandrasekhar1949} without the aid of electronic computers. \cite{Shawagfeh1993} used a coalition of the Adomian decomposition method and the Pad\'{e} approximants method to accelerate the convergence of the power series. A variational approach for solving Lane-Emden equations was presented by \cite{He2003}. \cite{Ramos2004} handled the Lane-Emden equation using a piecewise quasilinearization method that produces piecewise analytical solutions calculated through the analytic integration of the reduced linear constant-coefficients ordinary differential equation. A rational Legendre tau method was employed in the same year by \cite{Parand2004} for solving special initial-value problems of Lane-Emden type equations on a semi-infinite interval. \cite{Momoniat2006} applied Lie group analysis to achieve a larger radius of convergence than the power series solution of nonlinear Lane-Emden equations with $f(x,y) = e^y$ and $\alpha_{2} = 2$. Differential transformation method based on Taylor series expansion for solving singular initial-value problems of Lane-Emden type equations was introduced by \cite{Erturk2007}. \cite{Dehghan2008} studied the Lane-Emden equation using variational iteration method. \cite{Marzban2008} applied a hybridization of block-pulse functions and Lagrange interpolating polynomials to reduce the computation of nonlinear initial-value problems to a system of non-algebraic equations. \cite{Yang2010} used a spectral method based on a truncated series of Chebyshev expansions to solve singular initial value problems of Lane-Emden type equations. Bernstein polynomial approximations were employed by \cite{Pandey2012} to solve linear and nonlinear Lane-Emden equations. \cite{bhrawy2012jacobi} treated singular initial-value problems of nonlinear Lane-Emden type equations defined on a semi-infinite domain using shifted Jacobi polynomials for the spatial approximation integrated with shifted Jacobi-Gauss points as collocation nodes. \cite{yuzbacsi2013improved} presented a predictor-corrector algorithm to solve linear Lane-Emden equations provided with certain mixed conditions. In particular, the predicted solution is computed using a Bessel collocation method, and then correctified by solving an additional error problem that is constructed using a residual error function. \cite{Ozturk2014} applied a collocation method based on Hermite polynomial approximations. \cite{Smarda2015} presented a computational method integrating both the differential transformation method and a modified general formula for the Adomian polynomials for solving singular initial-value problems of nonlinear Lane-Emden type equations. Recently, \cite{Calvert2016} used modified rational Bernoulli function approximations to solve nonlinear Lane-Emden equations on the semi-infinite domain. 

To avoid the ill-conditioning of differential operators, an alternative direction to the aforementioned methods is to recast the differential equation into its integral formulation to take advantage of the well-conditioning of integral operators, then discretize the latter using various discretization techniques. Perhaps the first approach in this direction for solving Lane-Emden equations was presented by \cite{Yousefi2006} in 2006, using Legendre wavelet approximations and the Gaussian integration method. A similar approach was later presented by \cite{KarimiVanani2010} in which the transformed integral equation was eventually converted into Pad\'{e} series form. \cite{wazwaz2013adomian} applied the Adomian decomposition method to handle the integral form of the Lane-Emden equations. The Adomian decomposition method was later applied by \cite{Rach2014} on some alternate derivations for the Volterra integral forms of the Lane-Emden equation.  

Motivated by the abundant advantages and large success of pseudospectral methods in broad scientific areas and applications \cite{fornberg1998,boyd2001chebyshev,Mason2002,Elgindy2017,Suh2009,Ng2003,Xu2016,Elgindy2016,Xiao2016}, the exceedingly accurate and stable approximations obtained via integral reformulations \cite{Du2016,elgindy2016high,elgindy2013solving,Elgindy2013,Greengard1991spectral}, and the useful fundamental properties of Gegenbauer basis polynomials \cite{Elgindy2013a,Keiner2009computing}, we present the shifted Gegenbauer integral pseudospectral method (SGIPSM): a high-order, efficient, and exponentially convergent pseudospectral method to numerically solve Problems 1 and 2. The proposed method employs shifted Gegenbauer polynomial expansions to directly discretize the integral forms of Problems 1 and 2 on $(0,b]$, and adopts collocations in the nodal space at the shifted flipped-Gegenbauer-Gauss-Radau (SFGGR) points to overcome the singularity of the Lane-Emden equation at the origin that reflects the main difficulty, and to accurately impose the boundary condition at $x = b$. Moreover, instead of directly integrating both sides of Eqs. \eqref{eq:1} and \eqref{eq:4}, the proposed method follows closely the method of \cite{Elgindy2013} by introducing a useful substitution to the second-order solution derivative, solving Problems 1 and 2 in terms of the new substituting function, and then stably recovering the original solution function via successive integration. All necessary integral evaluations are accomplished using some newly developed SFGGR-based shifted Gegenbauer integration matrices. This approach entails the approximation of the 1- and 2-fold integrals of the solution and substituting functions, whereas the former approach may require the evaluation of the 1- and 2-fold  integrals of highly nonlinear and intricate integrands. 
 
The rest of the article is organized as follows: in Section \ref{sec:Pre1}, we begin by briefly reviewing some basic preliminaries concerned with Gegenbauer polynomials. In Section \ref{sec:TGIM1}, we show how to construct the SFGGR-based shifted Gegenbauer integration matrices from the standard Gegenbauer integration matrices in addition to some error analysis results. In Section \ref{sec:TDOLEE1}, we introduce the SGIPSM for solving Problems 1 and 2. Section \ref{sec:EBACR1} is devoted for the study of error bounds and convergence analysis of the SGIPSM. Section \ref{sec:NE1} is designed for numerical simulations and comparisons in which five numerical test examples are studied followed by some final concluding remarks and a discussion in Section \ref{Conc}. Appendix \ref{app:Alg1} is reserved for the presentation of four efficient computational algorithms to implement the SGIPSM.   

\section{Preliminaries}
\label{sec:Pre1}
In this section, we present some prior properties of the Gegenbauer polynomials. We rely heavily on the useful standardization given by \cite[Eq. (A.2)]{elgindy2013optimal}.
 
The Gegenbauer polynomial, $G_{n}^{(\alpha)}(x)$, of degree $n \in \mathbb{Z}^{+}_0 = \mathbb{Z}^{+} \cup \{ 0\}$, and associated with the parameter $\alpha>-1/2$, is a real-valued function that appears as an eigensolution to the singular Sturm-Liouville problem in the interval $[-1,1]$ \cite{elgindy2016high}. Chebyshev polynomials of the first kind and Legendre polynomials are special cases of this family of orthogonal polynomials for $\alpha=0$ and $0.5$, respectively. Gegenbauer polynomials can be easily generated by the following recurrence relation
\begin{equation}\label{eq:6}
(n + 2\alpha)\,G_{n+1}^{(\alpha)}(x) = 2\,(n+\alpha)\,x\,G_{n}^{(\alpha)}(x) - n\,G_{n-1}^{(\alpha)}(x),\quad n = 1, 2, \ldots,
\end{equation}
starting with $G_{0}^{(\alpha)}(x) = 1$ and $G_{1}^{(\alpha)}(x) = x$. The leading coefficient, $K_n^{(\alpha)}$, of $G_{n}^{(\alpha)}(x)$ is defined by 
\begin{equation}\label{eq:9}
K_{n}^{(\alpha)} = 2^{n-1} \frac{\Gamma(n+\alpha) \Gamma(2 \alpha+1)}{\Gamma(n+2 \alpha)\Gamma(\alpha+1)},\quad n = 0, 1, \ldots.
\end{equation}
Gegenbauer polynomials are orthogonal with respect to the weight function $w^{(\alpha)}(x) = \left(1-x^{2}\right)^{\alpha-\frac{1}{2}}$, and their orthogonality relation is defined by, 
\begin{equation}\label{eq:10}
\int_{-1}^{1} G_{m}^{(\alpha)}(x)\,G_{n}^{(\alpha)}(x)\,w^{(\alpha)}(x)\,dx = \lambda_{n}^{(\alpha)}\delta_{m,n},
\end{equation}
where 
\begin{equation}\label{eq:11}
\lambda_{n}^{(\alpha)} = \frac{2^{2\alpha-1} n! \Gamma^{2}{(\alpha+\frac{1}{2})}}{(n+\alpha)\Gamma(n+2 \alpha)},
\end{equation}
is the normalization factor, and $\delta_{m,n}$ is the Kronecker delta function. We denote the zeros of the $(n+1)$th-degree polynomial $q_n^{(\alpha)}(x) = G_{n+1}^{(\alpha)}(x) - G_{n}^{(\alpha)}(x)$ (aka FGGR nodes) sorted in descending order by $x_{n,k}^{(\alpha)}, k = 0, 1,\ldots, n$, and denote their set by $\mathbb{S}_{n}^{(\alpha)}$. The orthonormal Gegenbauer basis polynomials are defined by
\begin{equation}\label{eq:12}
\phi_{j}^{(\alpha)}(x) = \frac{G_{j}^{(\alpha)}(x)}{\sqrt{\lambda_{j}^{(\alpha)}}},\quad j = 0, \ldots, n,
\end{equation}
and they satisfy the discrete orthonormality relation
\begin{equation}\label{eq:13}
\sum \limits_{j=0}^{n} \varpi_{n,j}^{(\alpha)}\, \phi_{s}^{(\alpha)} \left(x_{n,j}^{(\alpha)}\right) \phi_{k}^{(\alpha)} \left(x_{n,j}^{(\alpha)}\right) = \delta_{sk},\quad s,k = 0, \ldots, n,
\end{equation}
where $\varpi_{n,j}^{(\alpha)}, j = 0, 1, \ldots, n$, are the corresponding Christoffel numbers of the FGGR quadrature formula on the interval $[-1,1]$ that are defined by
\begin{subequations}\label{subeq:14}
\begin{align}
\varpi_{n,0}^{(\alpha)} &= \left(\alpha +\frac{1}{2}\right) \aleph_{n,0}^{\left(\alpha\right)},\\
 \varpi_{n,j}^{(\alpha)} &= \aleph_{n,j}^{\left(\alpha\right)},\quad j = 1, 2, \ldots, n,
\end{align}
\end{subequations}
\begin{equation}\label{eq:15}
\aleph_{n,j}^{\left(\alpha\right)} = 2^{2 \alpha-1} \frac{\Gamma^{2}{\left(\alpha + \frac{1}{2}\right)} n!}{\left(n+\alpha+\frac{1}{2}\right) \Gamma{(n+2\alpha+1)}} \left(1+x_{n,j}^{(\alpha)}\right) \left(G_{n}^{(\alpha)}\left(x_{n,j}^{(\alpha)}\right)\right)^{-2},\quad j = 0, 1, \ldots, n.
\end{equation}
\section{The SFGGR-based shifted Gegenbauer integration matrix}
\label{sec:TGIM1}
Suppose that we approximate a real-valued function $f(x) \in C^{n+1}[-1,1]$ by the following truncated series expansion of Gegenbauer basis polynomials
\begin{equation}\label{eq:18}
f(x) \approx P_nf(x) = \sum\limits_{i = 0}^{n} {{a_i}\,G_{i}^{(\alpha )}(x)},
\end{equation}
where $a_i, i = 0, \ldots, n,$ are the Gegenbauer spectral coefficients. Denote the function values evaluated at the FGGR nodes, $f\left(x_{n,i}^{(\alpha)}\right)$, by $f_{n,i}^{(\alpha)}$, for all $i$. The first-order Gegenbauer integration matrix calculated at the FGGR nodes is a linear map, $\mathbf{Q}^{(1)}$, which takes a vector of $(n + 1)$ function values, ${\bm{F}} = {\left( {{f_{n,0}^{(\alpha)}},{f_{n,1}^{(\alpha)}}, \ldots ,{f_{n,n}^{(\alpha)}}} \right)^T}$, to a vector of $(n + 1)$ integral values
\[\bm{I}_n^{(\alpha)} = {\left( {\int_{-1}^{x_{n,0}^{(\alpha )}} {{P_n}f(x)\,dx} ,\int_{-1}^{x_{n,1}^{(\alpha )}} {{P_n}f(x)\,dx} , \ldots ,\int_{-1}^{x_{n,n}^{(\alpha )}} {{P_n}f(x)\,dx} } \right)^T},\]
such that
\begin{equation}
	\bm{I}_n^{(\alpha)} = \mathbf{Q}^{(1)}\,{\bm{F}}.
\end{equation}

Following the method presented by \cite{elgindy2013optimal}, one can construct the elements of $\mathbf{Q}^{(1)}, 0\leq i,k \leq n,$ by the following theorem.
\begin{thm}\label{th:1} 
Let $f(x) \in C^{n+1}[-1,1]$ be approximated by the Gegenbauer interpolant $P_nf$ given by Eq. \eqref{eq:18}, where the coefficients $a_i,\;i = 0, \ldots ,n,$ are obtained by interpolating $f$ at the FGGR points. Then there exist  a first-order integration  matrix $\mathbf{Q}^{(1)}=\left(Q_{i,j}^{(1)}\right),0\le i,j\le n$, and some numbers $\xi_{n,i}^{(1)} \in (-1,1)$, $i=0,\ldots,n,$ such that
\begin{equation}\label{eq:19}
\int_{-1}^{x_{n,i}^{(\alpha)}} f(x)\,dx= \sum_{k=0}^{n} Q_{i,k}^{(1)}\,f_{n,k}^{(\alpha )} +{}_fE_n^{(1)}\left(x_{n,i}^{(\alpha)},\xi_{n,i}^{(1)}\right)\quad \forall i,
\end{equation}
where
\begin{equation}\label{eq:20}
Q_{i,k}^{(1)}= \sum_{j=0}^{n} \left(\lambda_{j}^{(\alpha)}\right)^{-1} \varpi_{n,k}^{(\alpha)} G_{j}^{(\alpha)}\left(x_{n,k}^{(\alpha)}\right) \int_{-1}^{x_{n,i}^{(\alpha)}} G_{j}^{(\alpha)}(x)dx\quad \forall i,k,
\end{equation}
\begin{equation}\label{eq:21}
{}_fE_{n}^{(1)}\left(x_{n,i}^{(\alpha)},\xi_{n,i}^{(1)}\right)=\frac{f^{(n+1)}\left(\xi_{n,i}^{(1)}\right)}{(n+1)!K_{n+1}^{(\alpha)}  } \int_{-1}^{x_{n,i}^{(\alpha)}} q_n^{(\alpha)}(x)\, dx \quad \forall i,
\end{equation}
and the weights $\varpi_{n,k}^{(\alpha)}$ are as defined by Eqs. \eqref{subeq:14}.
\end{thm}
\begin{proof}
 Since $f_{n,j}^{(\alpha )} = \sum \limits_{k=0}^{n}a_{k}\,G_{k}^{(\alpha)}\left(x_{n,j}^{(\alpha)}\right),\; j=0,\ldots,n$, then
\begin{align*}
\sum \limits_{j=0}^{n} \varpi_{n,j}^{(\alpha)}G_{s}^{(\alpha)}\left(x_{n,j}^{(\alpha)}\right) f_{n,j}^{(\alpha )} &= \sum \limits_{k=0}^{n} a_{k} \sum \limits _{j=0}^{n}\varpi_{n,j}^{(\alpha)}G_{s}^{(\alpha)}\left(x_{n,j}^{(\alpha)}\right)G_{k}^{(\alpha)}\left(x_{n,j}^{(\alpha)}\right),\\
&= \sum \limits_{k=0}^{n} a_{k} \sum \limits _{j=0}^{n}\varpi_{n,j}^{(\alpha)}\sqrt{\lambda_{s}^{(\alpha)}\lambda_{k}^{(\alpha)}}\phi_{s}^{(\alpha)}\left(x_{n,j}^{(\alpha)}\right)\phi_{k}^{(\alpha)}\left(x_{n,j}^{(\alpha)}\right),\\
&\mathop = \limits^{\text{Eq. }\eqref{eq:13}} \sum \limits_{k=0}^{n} a_{k} \sqrt{\lambda_{s}^{(\alpha)}\lambda_{k}^{(\alpha)}} \delta_{sk} = a_{s}\,\lambda_{s}^{(\alpha)},\quad s = 0, \ldots, n.
\end{align*}
\begin{equation}\label{eq:25}
\Rightarrow a_{s}=\frac{1}{\lambda_{s}^{(\alpha)}} \sum_{j=0}^{n}\varpi_{n,j}^{(\alpha)}\, G_{s}^{(\alpha)}\left(x_{n,j}^{(\alpha)}\right)\,f_{n,j}^{(\alpha )}\quad \forall s.
\end{equation}\label{eq:26}
Substituting Eq. \eqref{eq:25} into Eq. \eqref{eq:18} yields the Gegenbauer interpolant in Lagrange form as follows:
\begin{equation}
{P_n}f(x) = \sum_{k=0}^{n} f_{n,k}^{(\alpha )}\, L_{n,k}^{(\alpha)}(x),
\end{equation}
where
\begin{equation}\label{eq:27}
L_{n,k}^{(\alpha)}(x)= \varpi_{n,k}^{(\alpha)}\sum \limits_{j=0}^{n} \frac{1}{\lambda_{j}^{(\alpha)}} G_{j}^{(\alpha)}\left(x_{n,k}^{(\alpha)}\right) G_{j}^{(\alpha)}(x)\quad \forall k.
\end{equation}
Therefore, we can write
\begin{equation}\label{eq:28}
f(x) = \sum_{k=0}^{n} f_{n,k}^{(\alpha )}\, L_{n,k}^{(\alpha)}(x) + {}_fE_{n}\left(x,\xi \right)\quad \forall x \in [-1,1],
\end{equation}
for some $\xi \in (-1, 1)$, where ${}_fE_n$ is the interpolation truncation error at the FGGR points defined by
\begin{equation}\label{eq:29}
{}_fE_{n}\left(x,\xi \right)=\frac{f^{(n+1)}\left(\xi\right)}{(n+1)!}\prod_{k=0}^{n}\left(x-x_{n,k}^{(\alpha)}\right).
\end{equation}
The proof is established by realizing that $q_n^{(\alpha)}(x) = K_{n+1}^{(\alpha)}\prod_{k=0}^{n}\left(x-x_{n,k}^{(\alpha)}\right)$, and integrating Eq. \eqref{eq:28} on $\left[-1,x_{n,i}^{(\alpha )} \right]$, for all $i$.
\end{proof}

Let us denote the $q$th-order Gegenbauer integration matrix by $\mathbf{Q}^{(q)}$, its $i$th row by $\mathbf{Q}_{i}^{(q)}$, for all $i$, and its entries by $Q_{i,k}^{(q)}$, for all $i, k$. The following theorem provides a useful means to calculate the 2-fold integral of $f$ through the second-order Gegenbauer integration matrix generated directly from the first-order Gegenbauer integration matrix.
\begin{thm}\label{th:2}
Given the assumptions of Theorem \ref{th:1}, there exist a second-order integration matrix $\mathbf{Q}^{(2)} = \left(Q_{i,j}^{(2)}\right),0\le i,j\le n$, and some numbers $\xi_{n,i}^{(2)} \in (-1,1)$, $i=0,\ldots,n,$ such that 
\begin{equation}\label{eq:33}
\int_{-1}^{x_{n,i}^{(\alpha)}}\int_{-1}^{x} f(t)\, dt\, dx = \sum_{k=0}^{n} Q_{i,k}^{(2)}\,f_{n,k}^{(\alpha )} + {}_fE_{n}^{(2)}\left(x_{n,i}^{(\alpha)},\xi_{n,i}^{(2)}\right)\quad \forall i, t \in [ - 1,1],
\end{equation}
where
\begin{equation}\label{eq:34}
Q_{i,k}^{(2)}=\left(x_{n,i}^{(\alpha)}-x_{n,k}^{(\alpha)}\right) Q_{i,k}^{(1)} \quad \forall i,k,
\end{equation}
\begin{equation}\label{eq:35}
{}_fE_{n}^{(2)}\left(x_{n,i}^{(\alpha)},\xi_{n,i}^{(2)}\right)=\frac{\left(x_{n,i}^{(\alpha)}-\xi_{n,i}^{(2)}\right)f^{(n+1)}\left(\xi_{n,i}^{(2)}\right)-(n+1)f^{(n)}\left(\xi_{n,i}^{(2)}\right)}{(n+1)!K_{n+1}^{(\alpha)}} \int_{-1}^{x_{n,i}^{(\alpha)}}\,q_n^{(\alpha)}(x)\,dx\quad \forall i.
\end{equation}
\end{thm}
\begin{proof}
The proof can be easily derived using Cauchy's formula for repeated integration.
\end{proof}
To discretize the nonlinear Lane-Emden equation \eqref{eq:1} on the shifted domain $(0,b]$, the shifted forms of Gegenbauer basis polynomials and their associated integration matrices are needed. To this end, let $G_{b,n}^{(\alpha)}(x)$ denote the shifted Gegenbauer polynomial $G_{n}^{(\alpha)}(2 x/b-1)$, for all $x \in [0,b]$. We denote the SFGGR nodes by $x_{b,n,k}^{(\alpha)}, k = 0, 1, \ldots, n$, and denote their set by $\mathbb{S}_{b,n}^{(\alpha)}$. Clearly   
\begin{equation}\label{eq:16}
x_{b,n,k}^{(\alpha)} = \frac{b}{2}\,\left(x_{n,k}^{(\alpha)} + 1\right)\quad \forall k.
\end{equation}
Following the above convention, let $\mathbf{Q}_{b}^{(q)}, \mathbf{Q}_{b,i}^{(q)}$, and $Q_{b,i,k}^{(q)}$, for all $i$ and $k$ denote the $q$th-order shifted integration matrix, its $i$th row, and its entries, respectively, and denote $f\left(x_{b,n,i}^{(\alpha)}\right)$ by $f_{b,n,i}^{(\alpha)}$, for all $i$. The following two theorems highlight the generation of the first-and second-order shifted Gegenbauer integration matrices $\mathbf{Q}_{b}^{(1)}$ and $\mathbf{Q}_{b}^{(2)}$, and mark the truncation errors of their associated quadratures.
\begin{thm}\label{th:3} 
Let $f(x) \in C^{n+1}[0,b]$ be approximated by a truncated series expansion of shifted Gegenbauer basis polynomials of the following form:
\begin{equation}\label{eq:37}
f(x) \approx \sum\limits_{i = 0}^{n} a_i\,G_{b,i}^{(\alpha )}(x),
\end{equation}
where $a_i, i = 0, \ldots ,n,$ are the shifted Gegenbauer spectral coefficients obtained by interpolating $f$ at the SFGGR points. Then there exist a first-order integration matrix $\mathbf{Q}_{b}^{(1)} = \left(Q_{b,i,j}^{(1)}\right),\; 0\le i,j\le n$, and some numbers $\xi_{b,n,i}^{(1)} \in (0,b), i=0,\ldots,n,$ such that 
\begin{equation}\label{eq:38}
\int_{0}^{x_{b,n,i}^{(\alpha)}} f(x)\,dx= \sum_{k=0}^{n} Q_{b,i,k}^{(1)}\,f_{b,n,k}^{(\alpha)} + {}_fE_{b,n}^{(1)}\left(x_{b,n,i}^{(\alpha)},\xi_{b,n,i}^{(1)}\right)\quad \forall i,
\end{equation}
where
\begin{equation}\label{eq:39}
\mathbf{Q}_{b}^{(1)} = \frac{b}{2}\mathbf{Q}^{(1)},
\end{equation}
\begin{equation}\label{eq:40}
{}_fE_{b,n}^{(1)}\left( {x_{b,n,i}^{(\alpha )},\xi _{b,n,i}^{(1)}} \right) = \frac{{{f^{(n + 1)}}\left( {\xi _{b,n,i}^{(1)}} \right)}}{{(n + 1)!K_{b,n + 1}^{(\alpha )}}}\int_0^{x_{b,n,i}^{(\alpha )}} {{q_{b,n}^{(\alpha)}}} (x){\mkern 1mu} dx\quad \forall i,
\end{equation}
$q_{b,n}^{(\alpha)}(x) = q_n^{(\alpha)}\left(2x/b-1\right)$, for all $x \in (0,b)$, and $K_{b,n}^{(\alpha )} = {\left( {2/b} \right)^{n}}K_n^{(\alpha )}$ is the leading coefficient of the shifted Gegenbauer polynomial $G_{b,n}^{(\alpha )}(x)$.
\end{thm}
\begin{proof}
Using the change of variable $x = b\left(y+1\right)/2$, and Theorem \ref{th:1}, we find that
	\begin{equation}\label{eq:41}
	\int_{0}^{x_{b,n,i}^{(\alpha)}} f(x)\,dx = \frac{b}{2} \int_{-1}^{x_{n,i}^{(\alpha)}} \bar{f}(y)\,dy =\frac{b}{2}\left(\sum\limits_{k=0}^{n} Q_{i,k}^{(1)}\, \bar{f}\left(x_{n,k}^{(\alpha)}\right) + \frac{\bar{f}^{(n+1)}\left(\xi_{n,i}^{(1)}\right)}{(n+1)!K_{n+1}^{(\alpha)}  } \int_{-1}^{x_{n,i}^{(\alpha)}} q_n^{(\alpha)}(x)\, dx\right)\quad \forall i,
		\end{equation}
for some $\xi_{n,i}^{(1)} \in (-1,1)$, $i=0,\ldots,n$, where $\bar{f}(y)=f(b(y+1)/2)$, for all $y \in [-1,1]$, and $\xi_{b,n,i}^{(1)} = b\left(\xi_{n,i}^{(1)}+1\right)/2$, for all $i.$
\end{proof}
\begin{thm}\label{th:4}
Given the assumptions of Theorem \ref{th:3}, there exist a second-order integration matrix $\mathbf{Q}_{b}^{(2)}=\left(Q_{b,i,j}^{(2)}\right),\; 0\le i,j\le n$, and some numbers $\xi_{b,n,i}^{(2)} \in (0,b)$,\; $i=0,\ldots,n,$ such that    
\begin{equation}\label{eq:44}
\int_{0}^{x_{b,n,i}^{(\alpha)}}\int_{0}^{x} f\left(t\right)\, dt\, dx = \sum_{k=0}^{n} Q_{b,i,k}^{(2)}\,f_{b,n,k}^{(\alpha)} + {}_fE_{b,n}^{(2)}\left(x_{b,n,i}^{(\alpha)},\xi_{b,n,i}^{(2)}\right)\quad \forall i, t \in [0,b],
\end{equation}
where
\begin{equation}\label{eq:45}
Q_{b,i,k}^{(2)} = \left(x_{b,n,i}^{(\alpha)}-x_{b,n,k}^{(\alpha)}\right) Q_{b,i,k}^{(1)}\quad \forall i, k,
\end{equation}
\begin{equation}\label{eq:46}
{}_fE_{b,n}^{(2)}\left( {x_{b,n,i}^{(\alpha )},\xi _{b,n,i}^{(2)}} \right) = \frac{{\left( {x_{b,n,i}^{(\alpha )} - \xi _{b,n,i}^{(2)}} \right){f^{(n + 1)}}\left( {\xi _{b,n,i}^{(2)}} \right) - (n + 1){f^{(n)}}\left( {\xi _{b,n,i}^{(2)}} \right)}}{{(n + 1)!K_{b,n + 1}^{(\alpha )}}}\int_0^{x_{b,n,i}^{(\alpha )}} {{q_{b,n}^{(\alpha)}}} (x)\, dx\quad \forall i.
\end{equation}
\end{thm}
\begin{proof}
The proof follows similar to that of Theorem \ref{th:2}.
\end{proof}
In the rest of the article, by $f(\bm{x})$, we mean ${[f({x_1}),f({x_2}), \ldots ,f({x_n})]^T}$, for any real-valued function $f: \bm{\Omega}_1 \subseteq \mathbb{R} \to \mathbb{R}$, and $\bm{x} = [x_1, x_2, \ldots, x_n]^T \in \mathbb{R}^n$. Similarly, by $g(\bm{x}, \bm{y})$, we mean ${[g({x_1},y_1),g({x_2},y_2), \ldots ,g({x_n},y_n)]^T}$, for any real-valued function $g: \bm{\Omega}_2 \subseteq \mathbb{R}^2 \to \mathbb{R}$, and $\bm{y} = [y_1, y_2, \ldots, y_n]^T \in \mathbb{R}^n$.
\section{The SGIPSM}
\label{sec:TDOLEE1}
In this section, we present the SGIPSM to solve numerically Problems 1 and 2. The derivation of the proposed SGIPSM is divided into two cases according to whether $\beta \neq 0$ or $\beta = 0$; therefore, we refer to these two cases by Assumptions 1 and 2, respectively.

We commence our numerical scheme by recasting Problem 1 into its integral formulation. By introducing the useful substitution
\begin{equation}\label{eq:47}
{y''}(x) = \phi(x)\quad \forall x \in (0,b],
\end{equation}
for some unknown function $\phi (x)$, we can recover the solution function and its derivative in terms of $\phi (x)$ by recursive integration; in particular
\begin{align}
{y'}(x) & =\alpha_{1}+\int_{0}^{x} \phi(t)\,dt,\label{eq:48}\\
y(x) &= y(0) + \alpha_{1} x+\int_{0}^{x} \int_{0}^{t_{2}} \phi(t_{1})\,dt_{1}\,dt_{2},\quad t_i \in (0,b]\;\forall i.\label{eq:49}
\end{align}
Substituting Eqs. \eqref{eq:47}, \eqref{eq:48}, and \eqref{eq:49} into Eq. \eqref{eq:1} yields the integral nonlinear Lane-Emden equation as follows:
\begin{equation}\label{eq:50}
\phi(x)+\frac{\alpha_{2}}{x} \left(\alpha_{1}+\int_{0}^{x} \phi(t)\,dt \right)+f\left(x,y(0)+\alpha_{1} x+\int_{0}^{x} \int_{0}^{t_{2}} \phi(t_{1})\,dt_{1}\,dt_{2}\right)=0 \quad \forall x \in (0,b].
\end{equation}
If Assumption 1 holds, then we can approximate $y(0)$ using Eqs. \eqref{eq:3}, \eqref{eq:48}, and \eqref{eq:49} as follows:
\begin{equation}\label{eq:51}
	y(0) \approx \frac{1}{\beta}\left(\delta - \gamma \left(\alpha_{1} + \mathbf{Q}_{b,0}^{(1)}\, \bm{\Phi}\right) \right) - \alpha_{1}b - \mathbf{Q}_{b,0}^{(2)}\,\bm{\Phi},
	\end{equation}
where $\bm{\Phi} = \phi\left(\bm{x}_{b,n}^{(\alpha)}\right)$, and\;$ {\bm{x}_{b,n}^{(\alpha)}}=\left(x_{b,n,0}^{(\alpha)},x_{b,n,1}^{(\alpha)}, \ldots,x_{b,n,n}^{(\alpha)}\right)^{T}$. Let $\mathbf{D}_{\bm{v}}= \diag \left(v_{0}, \ldots,v_{n}\right)$, for any vector $\bm{v}=\left(v_{i}\right)\in \mathbb{R}^{n+1},\;i=0, \ldots,n$, and denote the vector of reciprocals $\left(1/v_{0}, \ldots,1/v_{n}\right)^{T}$ by $\bm{v}^{\div}$. Moreover, let $\bm{1}_{n}$ and $\mathbf{I}_{n}$ denote the all ones column vector and the identity matrix each of size $n$, respectively. Then collocating the integral nonlinear Lane-Emden Eq. \eqref{eq:50} at the SFGGR nodes yields the following (n+1)th-order nonlinear system
\begin{subequations}
\begin{equation}\label{eq:52}
\mathbf{H}_{b,n}^{(\alpha)}\, \bm{\Phi} + f\left(\bm{x}_{b,n}^{(\alpha)},\bar{\bm{x}}_{b,n}^{(\alpha)} + \bm{\Theta}_{b,n}\, \bm{\Phi}\right) =-\alpha_{1} \alpha_{2} \left(\bm{x}_{b,n}^{(\alpha)}\right)^{\div},
\end{equation}
where
\begin{equation}\label{eq:53}
\mathbf{H}_{b,n}^{(\alpha)}=\mathbf{I}_{n+1} + \alpha_{2} \mathbf{D}_{\left(\bm{x}_{b,n}^{(\alpha)}\right)^{\div}}\,\mathbf{Q}_{b}^{(1)},
\end{equation}
\begin{equation}\label{eq:54} 
\bar{\bm{x}}_{b,n}^{(\alpha)}=\frac{1}{\beta} \left(\delta-\gamma \alpha_{1} \right) \bm{1}_{n+1}+\alpha_{1} \left(\bm{x}_{b,n}^{(\alpha)}-b\bm{1}_{n+1}\right),
\end{equation}
\begin{equation}\label{eq:55}
\mathbf{\Theta }_{b,n}=\left(\mathbf{Q}_{b}^{(2)}-\left(\mathbf{Q}_{b,0}^{(2)}+\frac{\gamma}{\beta} \mathbf{Q}_{b,0}^{(1)}\right) \otimes \bm{1}_{n+1}\right),
\end{equation}
\end{subequations}
and ``$\otimes$'' denotes the Kronecker product. The solution of the nonlinear system \eqref{eq:52} produces the approximate solution vector $\bm{\Phi}$. The original solution vector, $y\left(\bm{x_{b,n}^{(\alpha)}}\right)$, can then be recovered using Eq. \eqref{eq:49} as follows: 
\begin{equation}\label{eq:56}
y\left(\bm{x_{b,n}^{(\alpha)}}\right) \approx \bar{\bm{x}}_{b,n}^{(\alpha)} + \mathbf{\Theta}_{b,n}\, \bm{\Phi}.
\end{equation} 
We can find the approximate solution at any point $x \in (0,b]$ using interpolation at the SFGGR nodes as follows:
\begin{equation}\label{eq:57}
y(x) \approx \sum_{k=0}^{n} y_{b,n,k}^{(\alpha)}\, L_{b,n,k}^{(\alpha)}(x)\quad \forall x \in (0,b],
\end{equation}
where $y_{b,n,k}^{(\alpha)} = y\left(x_{b,n,k}^{(\alpha)}\right)$, for all $k$,  
\begin{equation}
	L_{b,n,k}^{(\alpha )}(x) = \varpi _{b,n,k}^{(\alpha )}\sum\limits_{j = 0}^n {\frac{1}{{\lambda _{b,j}^{(\alpha )}}}} G_{b,j}^{(\alpha )}\left( {x_{b,n,k}^{(\alpha )}} \right)G_{b,j}^{(\alpha )}(x)\quad \forall k,
\end{equation}
is the $n$th-degree shifted Lagrange interpolating polynomial, $\lambda _{b,j}^{(\alpha )} = {\left( {b/2} \right)^{2\,\alpha }}\lambda _j^{(\alpha )}$, and $\varpi _{b,n,j}^{(\alpha )} = {\left( {b/2} \right)^{2\,\alpha }}\varpi _{n,j}^{(\alpha )}$, for all $j$. This technique can be carried out using Algorithm \ref{algorithm:1} in \ref{app:Alg1}. For Problem 2, Eqs. \eqref{eq:52} can be reduced into the following linear algebraic system
\begin{subequations}
\begin{equation}\label{eq:58}
\mathcal{A}\,{\bm{\Phi}} = \mathcal{B},
\end{equation}
where
\begin{align}
\mathcal{A} &= \mathbf{H}_{b,n}^{(\alpha)} + \mathbf{D}_{p\left(\bm{x_{b,n}^{(\alpha)}}\right)}\,\mathbf{\Theta }_{b,n},\label{eq:CoeffAB1}\\
\mathcal{B} &= g\left(\bm{x}_{b,n}^{(\alpha)}\right)-\alpha_{1} \alpha_{2}\left(\bm{x}_{b,n}^{(\alpha)}\right)^{\div} -\mathbf{D}_{p\left(\bm{x_{b,n}^{(\alpha)}}\right)}\, \bar{\bm{x}}_{b,n}^{(\alpha)}.\label{eq:CoeffAB2}
\end{align}
\end{subequations}
The original solution vector $y\left(\bm{x_{b,n}^{(\alpha)}}\right)$ can again be recovered using Eq. \eqref{eq:56}. This technique can be implemented using Algorithm \ref{algorithm:2} in \ref{app:Alg1}.

Now consider the second scenario when Assumption 2 holds. In this case, we can readily impose the boundary condition \eqref{eq:3} via Eq. \eqref{eq:48} by the following approximation:
\begin{equation}\label{eq:8new}
	 \mathbf{Q}_{b,0}^{(1)}\, \bm{\Phi} \approx \frac{\delta}{\gamma} - \alpha_{1}.
	\end{equation}
Collocating Eq. \eqref{eq:50} at the SFGGR nodes and combining the result with Eq. \eqref{eq:8new} gives the following $(n+2)$th-order nonlinear system in the $(n+2)$ unknowns $\bm{\Phi}$ and $y(0)$:
\begin{subequations}
\begin{equation}\label{eq:9new}
{\mathbf{\hat H}}_{b,n}^{(\alpha )}\bm{\Phi}  + \hat {f}_{b,n}^{(\alpha )}(\bm{\Phi}) = {\bm{\hat x}}_{b,n}^{(\alpha )},
\end{equation}
where
\begin{equation}\label{eq:hatH1}
{\mathbf{\hat H}}_{b,n}^{(\alpha )} = \left[ {{\mathbf{H}}_{b,n}^{(\alpha )};{\mathbf{Q}}_{b,0}^{(1)}} \right],
\end{equation}
\begin{equation}\label{eq:fbn1}
	\hat {f}_{b,n}^{(\alpha )}(\bm{\Phi}) = \left[ {f\left( {{\bm{x}}_{b,n}^{(\alpha )}, y(0)\,{{\bm{1}}_{n + 1}} + {\alpha _1}\,{\bm{x}}_{b,n}^{(\alpha )} + {\mathbf{Q}}_b^{(2)}\,\bm{\Phi} } \right);0} \right],
\end{equation}
\begin{equation}\label{hatx1}
	{\bm{\hat x}}_{b,n}^{(\alpha )} = \left[ { - {\alpha _1}\,{\alpha _2}\,{{\left( {{\bm{x}}_{b,n}^{(\alpha )}} \right)}^ \div };\frac{\delta }{\gamma } - {\alpha _1}} \right],
\end{equation}
\end{subequations}
and ``$[.;.]$'' is the vertical matrix concatenation along columns defined by ``${[{.^T},{.^T}]^T}$.'' The solution of the nonlinear system \eqref{eq:9new} produces the approximate solution vector $\bm{\Psi} = \left[\bm{\Phi};y(0)\right]$ of Problem 1. The original solution vector, $y\left(\bm{x_{b,n}^{(\alpha)}}\right)$, can then be recovered using Eq. \eqref{eq:49} as follows: 
\begin{equation}\label{eq:14new}
y\left(\bm{x_{b,n}^{(\alpha)}}\right) \approx y(0)\,\bm{1}_{n+1}+ \alpha_{1} \, \bm{x}_{b,n}^{(\alpha)} + \mathbf{Q}_{b}^{(2)}\, \bm{\Phi}.
\end{equation} 
This technique can be achieved using Algorithm \ref{algorithm:3} in \ref{app:Alg1}. For Problem 2, we can easily show after some mathematical manipulation that Eqs. \eqref{eq:9new} can be reduced into the following linear algebraic system
\begin{subequations}\label{eq:58new}
\begin{equation}\label{eq:58newC1}
\mathcal{C}\,{\bm{\Psi}} = \mathcal{D},
\end{equation}
where
\begin{align}
{\mathcal{C}} &= \left[ {\begin{array}{*{20}{c}}
{{\mathbf{H}}_{b,n}^{(\alpha )} + {{\mathbf{D}}_{p\left( {{\bm{x}}_{b,n}^{(\alpha )}} \right)}}\,{\mathbf{Q}}_b^{(2)}}&{p\left( {{\bm{x}}_{b,n}^{(\alpha )}} \right)}\\
{{\mathbf{Q}}_{b,0}^{(1)}}&0
\end{array}} \right],\label{eq:C1}\\
{\mathcal{D}} &= {\bm{\hat g}}_{b,n}^{(\alpha )} + {\bm{\hat x}}_{b,n}^{(\alpha )},\label{eq:D1}\\
{\bm{\hat g}}_{b,n}^{(\alpha )} &= \left[ {g\left( {{\bm{x}}_{b,n}^{(\alpha )}} \right) - {\alpha _1}\,{{\mathbf{D}}_{p\left( {{\bm{x}}_{b,n}^{(\alpha )}} \right)}}\,{\bm{x}}_{b,n}^{(\alpha )};0} \right].\label{eq:g1}
\end{align}
\end{subequations}
The original solution vector, $y\left(\bm{x_{b,n}^{(\alpha)}}\right)$, can again be recovered using Eq. \eqref{eq:14new}. This technique can be accomplished using Algorithm \ref{algorithm:4} in \ref{app:Alg1}. 
\section{Error bounds and convergence analysis}
\label{sec:EBACR1}
We begin this section by stating one technical lemma and three theorems that are needed for the subsequent derivation of the error estimates of the SGIPSM. Let $\mathbb{R}_0^+ = \mathbb{R}^+ \cup \{ 0\}$, and $\|f\|_{\infty,\mathbb{S}}=\sup\left\{\,\left|f(x)\right|:x\in \mathbb{S}\,\right\}$, be the uniform norm (or sup norm) of a real-valued bounded function $f$ defined on a set $\mathbb{S}\subset \mathbb{R}$. The following lemma is a modified version of \cite[Lemma 4.1]{elgindy2016high}. 
\begin{lem}\label{lem:5}
The uniform norm of the $n$th-degree Gegenbauer polynomial, ${G_n^{(\alpha )}}$, and its associated shifted form, ${G_{b,n}^{(\alpha )}}$, is given by
$$ {\left\| {G_n^{(\alpha )}} \right\|_{\infty ,[ - 1,1]}}=\left\| {G_{b,n}^{(\alpha )}} \right\|_{\infty ,[0,b]} =
\begin{cases}
1,\quad n \geq 0,\;  \alpha \geq 0, \\
\displaystyle{\binom{\alpha+\frac{n}{2}-1}{\frac{n}{2}} \frac{n! \Gamma{(2 \alpha)}}{\Gamma{(n+2\alpha)}}}, \quad \frac{n}{2} \in \mathbb{Z}_0^ +  \wedge \frac{-1}{2}<\alpha < 0,\\
\displaystyle{\frac{n \Gamma{\left(\alpha+\frac{1}{2}\right)} \Gamma{\left(\frac{n}{2}\right)}  }{\sqrt{\pi} \sqrt{n\,(2 \alpha+n)}\, \Gamma{\left(\frac{n}{2}+\alpha\right)}}} ,\quad \frac{{n + 1}}{2} \in {\mathbb{Z}^ + } \wedge \frac{-1}{2}<\alpha < 0, \\
D^{(\alpha)}n^{-\alpha}, \quad \frac{-1}{2}<\alpha<0,\; n\rightarrow \infty,\\
\end{cases}
$$
where $D^{(\alpha)} > 1$ is a constant dependent on $\alpha$, but independent of $n$.
\end{lem}
The following three theorems highlight the error bounds of the Gegenbauer quadratures associated with $\mathbf{Q}^{(1)}, \mathbf{Q}_b^{(1)}$, and $\mathbf{Q}_b^{(2)}$, respectively.
\begin{thm}
\label{th:6}
Assume that $f(x) \in C^{n+1}[-1,1]$ and $\left\| f^{(n + 1)} \right\|_{\infty ,[-1,1]} = A \in \mathbb{R}^{+}$, for some constant $A$ independent of $n$. Moreover, let $\int_{-1}^{x_{n,i}^{(\alpha)}}f(x)\,dx$ be approximated by the Gegenbauer quadrature, $\sum\nolimits_{k = 0}^n {Q_{i,k}^{(1)}\,f_{n,k}^{(\alpha )}},\; i = 0, \ldots ,n.$ Then there exist some positive constants ${B^{(\alpha )}}$ and ${C^{(\alpha )}}$ dependent on $\alpha$ and independent of n such that the quadrature truncation error is bounded by the following inequalities 
\begin{equation}\label{eq:59}
\left|{}_fE_{n}^{(1)}\left(x_{n,i}^{(\alpha)},\xi_{n,i}^{(1)}\right)\right| \leq \frac{A \Gamma{(n+2\alpha+1)} \Gamma{(\alpha+1)} \left(x_{n,i}^{(\alpha)}+1\right) }{2^{n} (n+1)! \Gamma{(n+\alpha+1)} \Gamma{(2 \alpha+1)}}{\left\| q_n^{(\alpha)} \right\|_{\infty ,[ - 1,1]}},\quad i = 0, \ldots, n,
\end{equation}
where
\begin{equation}\label{eq:60}
{\left\| q_n^{(\alpha)} \right\|_{\infty ,[-1,1]}}=
\begin{cases}
2, \quad n \geq 0 ,\;\alpha \geq 0, \\
\displaystyle{\frac{\Gamma{\left(\alpha+\frac{1}{2}\right)} \Gamma{\left(\frac{n+1}{2}\right)}}{\sqrt{\pi} \Gamma{\left(\alpha+\frac{n+1}{2}\right)}}\left(1+\sqrt{\frac{n+1}{2\alpha+n+1}}\right)},\quad \frac{n}{2} \in \mathbb{Z}_0^ +  \wedge \frac{-1}{2}<\alpha < 0,\\
\displaystyle{\frac{{\Gamma \left( {\alpha  + \frac{1}{2}} \right)\left( {\sqrt {n\left( {2\alpha  + n} \right)}  + n} \right)\Gamma \left( {\frac{n}{2}} \right)}}{{2 \sqrt \pi \, \Gamma \left( {\frac{n}{2} + \alpha  + 1} \right)}}}
,\quad \frac{{n + 1}}{2} \in {\mathbb{Z}^ + } \wedge \frac{-1}{2}<\alpha < 0.
\end{cases}
\end{equation}
Moreover, when $n \rightarrow \infty$, we have
\begin{equation}\label{ineq:asymineqalphanonneg1}
\left| {{}_fE_n^{(1)}\left( {x_{n,i}^{(\alpha )},\xi _{n,i}^{(1)}} \right)} \right| \le {B^{(\alpha )}}{2^{1- n}}{e^n}\frac{{1 + x_{n,i}^{(\alpha )}}}{{{n^{n + \frac{3}{2} - \alpha }}}},
\end{equation}
for $\alpha  \ge 0$, and
\begin{equation}\label{ineq:asymineqalphaneg1}
\left| {{}_fE_n^{(1)}\left( {x_{n,i}^{(\alpha )},\xi _{n,i}^{(1)}} \right)} \right|\mathop  < \limits_ \sim  {C^{(\alpha )}}{\left( {\frac{e}{2}} \right)^n}\frac{{1 + x_{n,i}^{(\alpha )}}}{{{n^{n + \frac{3}{2}}}}},
\end{equation}
for  $- 1/2 < \alpha  < 0$, where $\mathop  < \limits_ \sim$ means ``less than or asymptotically equal to.'' 
\end{thm}
\begin{proof}
Inequality \eqref{eq:59} is determined through Eqs. \eqref{eq:9} and \eqref{eq:21}, and Lemma \ref{lem:5}. For \,$n\rightarrow \infty$, we find through \cite[Lemma 4.2]{elgindy2016high} that
\begin{equation}
\left|{}_fE_{n}^{(1)}\left(x_{n,i}^{(\alpha)},\xi_{n,i}^{(1)}\right)\right| \leq \frac{B^{(\alpha)}\left(1+x_{n,i}^{(\alpha)}\right)}{n^{\frac{3}{2}-\alpha} (\frac{2n}{e})^{n}}  \left \|q_{n}^{(\alpha)}\right \|_{\infty,[-1,1]},
\label{eq:63}
\end{equation}
where $B^{(\alpha)}$ is a constant depend on $\alpha$ and independent of $n$. Inequality \eqref{ineq:asymineqalphanonneg1} is established by combining this result with Lemma \ref{lem:5} such that
$$
\left|{}_fE_{n}^{(1)}\left(x_{n,i}^{(\alpha)},\xi_{n,i}^{(1)}\right)\right| \leq \frac{B^{(\alpha)} \left(1+x_{n,i}^{(\alpha)}\right)}{n^{\frac{3}{2}-\alpha}(\frac{2n}{e})^{n}}
\begin{cases}
2,\quad \alpha \geq 0, \\
D_{1}^{(\alpha)}(n+1)^{-\alpha}+D_{2}^{(\alpha)}n^{-\alpha},\quad \frac{-1}{2}<\alpha < 0, \\
\end{cases}
$$
where $D_1^{(\alpha)}, D_2^{(\alpha)} > 1$  are constants dependent on $\alpha$ and independent of $n$. The asymptotic inequality \eqref{ineq:asymineqalphaneg1} is derived by realizing that $n+1 \sim n$ as $n \to \infty$, and setting $C^{(\alpha)} = B^{(\alpha )} \left(D_{1}^{(\alpha )}+D_{2}^{(\alpha )}\right)$.
\end{proof}
\begin{thm}\label{th:7}
Assume that $f(x) \in C^{n+1}[0,b]$ and $\left\| f^{(n + 1)} \right\|_{\infty ,[0,b]} = A \in \mathbb{R}^{+}$, for some constant $A$ independent of $n$. Moreover, let $\int_{0}^{x_{b,n,i}^{(\alpha)}}f(x)\,dx$ be approximated by the shifted Gegenbauer quadrature, $\sum\limits_{k = 0}^n {Q_{b,i,k}^{(1)}{\kern 1pt}f_{b,n,k}^{(\alpha )}} ,\; i = 0, \ldots ,n$. Then there exist some positive constants $B^{(\alpha)}$ and $C^{(\alpha)}$ dependent on $\alpha$ and independent of n such that the quadrature truncation error is bounded by the following inequalities 
\begin{equation}\label{eq:64}
\left| {{}_fE_{b,n}^{(1)}\left( {x_{b,n,i}^{(\alpha )},\xi _{b,n,i}^{(1)}} \right)} \right| \le \frac{{{2^{- 2 n - 1 }}A\,{b^{n + 1}}\,x_{b,n,i}^{(\alpha )}\,\Gamma (n + 2\alpha  + 1)\Gamma (\alpha  + 1)}}{{(n + 1)!\Gamma (n + \alpha  + 1)\Gamma (2\alpha  + 1)}}{\left\| {q_{b,n}^{(\alpha )}} \right\|_{\infty ,[0,b]}},\quad i = 0, \ldots ,n.
\end{equation}
Moreover, when $n \rightarrow \infty$, we have
\begin{equation}
\left|{}_fE_{b,n}^{(1)}\left(x_{b,n,i}^{(\alpha)},\xi_{b,n,i}^{(1)}\right)\right| \leq
{B^{(\alpha )}}{b^{n + 1}}{\left(\frac{e}{4}\right)^{n}}\frac{{x_{b,n,i}^{(\alpha )}}}{{{n^{n + {\textstyle{3 \over 2}} - \alpha }}}},
\label{eq:inequalityb1}
\end{equation}
for $\alpha \ge 0$, and
\begin{equation}
\left|{}_fE_{b,n}^{(1)}\left(x_{b,n,i}^{(\alpha)},\xi_{b,n,i}^{(1)}\right)\right| \mathop  < \limits_ \sim
{C^{(\alpha )}}{2^{ - 2n - 1}}{e^n}{n^{ - n - \frac{3}{2}}}x_{b,n,i}^{(\alpha )}\,{b^{n + 1}},
\label{eq:inequalityb2}
\end{equation}
for  $-1/2 < \alpha  < 0$.
\end{thm}
\begin{proof}
The proof is straightforward using Eqs. \eqref{eq:9} and \eqref{eq:40}, Lemma \ref{lem:5}, and \cite[Lemma 4.2]{elgindy2016high}.
\end{proof}

\begin{thm}\label{th:8}
Assume that $f(x) \in C^{n+1}[0,b]$ and ${\left\| {{f^{(n + k)}}} \right\|_{\infty ,[0,b]}} = {A_k} \in {\mathbb{R}^ + }, k = 0,1,$ where the constants ${A_0}$ and ${A_1}$ are independent of $n$. Moreover, let $ \int_{0}^{x_{b,n,i}^{(\alpha)}} \int_{0}^{x}f\left(t\right)\, dt\, dx$ be approximated by the shifted Gegenbauer quadrature, $\sum\limits_{k = 0}^n {Q_{b,i,k}^{(2)}\,f_{b,n,k}^{(\alpha )}} ,\; i = 0, \ldots ,n.$ Then there exist some positive  constants $C^{(\alpha)}$ and $D^{(\alpha)}$ dependent on $\alpha$ and independent of $n$ such that the quadrature truncation error is bounded by the following inequalities 
\begin{equation}\label{eq:67}
\left|{}_fE_{b,n}^{(2)}\left(x_{b,n,i}^{(\alpha)},\xi_{b,n,i}^{(2)}\right)\right| \leq \frac{{{2^{ - 2n - 1}}\,x_{b,n,i}^{(\alpha )}\,\Gamma \left( {\alpha  + 1} \right)\,{b^{n + 1}}\,\Gamma \left( {n + 2\alpha  + 1} \right)\,\left( {A_0\, (n + 1) + b\,{\kern 1pt} A_1} \right)}}{{(n + 1)!\,\Gamma \left( {2\alpha  + 1} \right)\,\Gamma \left( {n + \alpha  + 1} \right)}}\,{\left\| {{q_{b,n}^{(\alpha)}}} \right\|_{\infty ,[0,b]}},\quad i = 0, \ldots ,n. 
\end{equation}
Moreover, when $n \rightarrow \infty$, we have
\begin{equation}
\left| {{}_fE_{b,n}^{(2)}\left( {x_{b,n,i}^{(\alpha )},\xi _{b,n,i}^{(2)}} \right)} \right| \le {b^{n + 1}}{C^{(\alpha)} }{\left( {\frac{{e}}{4}} \right)^n}{n^{ - \frac{3}{2} - n + \alpha }}x_{b,n,i}^{(\alpha )}\left( {{A_0}\,\left( {n + 1} \right) + b\,{A_1}} \right),
\label{eq:inequalitybb1}
\end{equation}
for $\alpha \ge 0$, and
\begin{equation}
\left|{}_fE_{b,n}^{(2)}\left(x_{b,n,i}^{(\alpha)},\xi_{b,n,i}^{(2)}\right)\right| \mathop  < \limits_ \sim
{2^{ - 2n - 1}}{b^{n + 1}}{D^{(\alpha )}}{{e}^n}{n^{ - n - \frac{3}{2}}}x_{b,n,i}^{(\alpha )}\left( {{A_0}\,(n + 1) + b\,{A_1}} \right),
\label{eq:inequalitybb2}
\end{equation}
for  $- 1/2 < \alpha  < 0$. 
\end{thm}
\begin{proof}
The proof is straightforward using Eqs. \eqref{eq:9} and \eqref{eq:46}, Lemma \ref{lem:5}, and \cite[Lemma 4.2]{elgindy2016high}.  
\end{proof}
Now suppose that we denote the exact and approximate solutions of each of Problems 1 and 2 by $y(x)$ and $y_{\text{app},n}(x)$, respectively. We are now ready to give the a priori and asymptotic a priori error estimates through the following key theorem, which proves the elegant exponential convergence rate of the SGIPSM for Problem 1. 
\begin{thm}[A priori and asymptotic a priori error estimates]\label{th:9}
Assume that $\left\|y^{(n+k+2)}\right\|_{\infty,(0,b]} = A_{k} \in \mathbb{R}^{+}, \;k=0,1$, and $y$ is approximated by the nodal truncated series expansion of shifted Gegenbauer basis polynomials \eqref{eq:57}, where the second-order derivative values $y''\left( {x_{b,n,i}^{(\alpha )}} \right), i = 0, \ldots ,n,$ are obtained by collocating the integral nonlinear Lane-Emden equation \eqref{eq:50} at the SFGGR points. Suppose also that the boundary conditions \eqref{eq:2} and \eqref{eq:3} and Assumption 1 hold. Then the solution error values, $y_{b,n,i}^{(\alpha )} - {y_{\text{app},n}}\left( {x_{b,n,i}^{(\alpha )}} \right),\; i=0,\ldots,n$, are bounded by the following inequalities    
\begin{align}
\left| {y_{b,n,i}^{(\alpha )} - {y_{\text{app},n}}\left( {x_{b,n,i}^{(\alpha )}} \right)} \right| \le 	\frac{{{2^{ - 2n - 1}}{{\left\| {q_{b,n}^{(\alpha )}} \right\|}_{\infty ,(0,b]}}\Gamma \left( {\alpha  + 1} \right){b^{n + 1}}\Gamma \left( {n + 2\alpha  + 1} \right)}}{{\Gamma \left( {2\alpha  + 1} \right)\Gamma \left( {n + 2} \right)\Gamma \left( {n + \alpha  + 1} \right)}} \times\nonumber\\ 
\left( {{A_1}b\left( {x_{b,n,i}^{(\alpha )} + \left| {\frac{\gamma }{\beta }} \right| + b} \right) + {A_0}\left( {n + 1} \right)\left( {x_{b,n,i}^{(\alpha )} + b} \right)} \right).\label{eq:Sol2016}
\end{align}
Moreover, when $n \rightarrow \infty$, the asymptotic solution error values are bounded by
\begin{equation}
	\left| {y_{b,n,i}^{(\alpha )} - {y_{\text{app},n}}\left( {x_{b,n,i}^{(\alpha )}} \right)} \right| \le \mu _{1,b}^{(\alpha )}\,{\left( {\frac{e}{4}} \right)^n}{b^{n + 1}}\left\{ \begin{array}{l}
{n^{ - n - \frac{3}{2} + \alpha }},\quad {A_0} = 0,\\
{n^{ - n - \frac{1}{2} + \alpha }},\quad {A_0} \ne 0,
\end{array} \right.\quad i = 0, \ldots ,n,
\end{equation}
for $\alpha \geq 0$, and
\begin{equation}
	\left| {y_{b,n,i}^{(\alpha )} - {y_{\text{app},n}}\left( {x_{b,n,i}^{(\alpha )}} \right)} \right| \le \mu _{2,b}^{(\alpha )}\,{2^{ - 2n - 1}}{e^n}\,{b^{n + 1}}\left\{ \begin{array}{l}
{n^{ - n - \frac{3}{2}}},\quad {A_0} = 0,\\
{n^{ - n - \frac{1}{2}}},\quad {A_0} \ne 0,
\end{array} \right.\quad i = 0, \ldots ,n,
\end{equation}
for $\frac{-1}{2}<\alpha < 0$, where $\mu _{1,b}^{(\alpha )}$ and $\mu _{2,b}^{(\alpha )}$ are positive constants dependent on $b$ and $\alpha$, and independent of $n$.
\end{thm}
\begin{proof}
Using Eqs. \eqref{eq:49} and \eqref{eq:51}, and Theorems \ref{th:7} and \ref{th:8}, we can easily show that 
\begin{equation}
	\left| {y_{b,n,i}^{(\alpha )} - {y_{\text{app},n}}\left( {x_{b,n,i}^{(\alpha )}} \right)} \right| \le \left| {\frac{\gamma }{\beta }{}_{y''}E_{b,n,0}^{(1),\alpha }} \right| + \left| {{}_{y''}E_{b,n,0}^{(2),\alpha }} \right| + \left| {{}_{y''}E_{b,n,i}^{(2),\alpha }} \right|\quad \forall i,
\end{equation}
where ${}_{y''}E_{b,n,i}^{(j), \alpha} = {}_{y''}E_{b,n}^{(j)}\left(x_{b,n,i}^{(\alpha)}, \xi_{b,n,i}^{(j)}\right)$, for some $\xi_{b,n,i}^{(j)} \in (0,b), i = 0,\ldots,n; j = 1, 2$. This completes the proof of the theorem.
\end{proof}
\noindent Theorem \ref{th:9} shows that 
\begin{equation}
	\mathop {\lim }\limits_{n \to \infty } {\left\| {y - {y_{\text{app},n}}} \right\|_{\infty ,(0,b]}} = 0,
\end{equation}
with an error upper bound given by
\begin{equation}\label{eq:erruppb1}
	{\left\| {y - {y_{\text{app},n}}} \right\|_{\infty ,(0,b]}} = O\left( {\left\{ \begin{array}{l}
{\left( {\frac{e}{4}} \right)^n}{b^{n + 1}}{n^{ - n - \frac{3}{2} + \alpha }},\quad \alpha  \ge 0 \wedge {A_0} = 0,\\
{\left( {\frac{e}{4}} \right)^n}{b^{n + 1}}{n^{ - n - \frac{1}{2} + \alpha }},\quad \alpha  \ge 0 \wedge {A_0} \ne 0,\\
{2^{ - 2n - 1}}{e^n}\,{b^{n + 1}}{n^{ - n - \frac{3}{2}}},\quad  - 1/2 < \alpha  < 0 \wedge {A_0} = 0,\\
{2^{ - 2n - 1}}{e^n}\,{b^{n + 1}}{n^{ - n - \frac{1}{2}}},\quad  - 1/2 < \alpha  < 0 \wedge {A_0} \ne 0
\end{array} \right.} \right).
\end{equation}
The following theorem marks the priori and the asymptotic priori error estimates of the solution derivative when Assumption 2 holds.
\begin{thm}[A priori and asymptotic a priori error estimates]\label{th:10}
Assume that $\left\|y^{(n+3)}\right\|_{\infty,(0,b]} = A \in \mathbb{R}^{+}$, and $y$ is approximated by the nodal truncated series expansion of shifted Gegenbauer basis polynomials \eqref{eq:57}, where the second-order derivative values $y''\left( {x_{b,n,i}^{(\alpha )}} \right), i = 0, \ldots ,n,$ are obtained by collocating the integral nonlinear Lane-Emden equation \eqref{eq:50} at the SFGGR points. Suppose also that the boundary conditions \eqref{eq:2} and \eqref{eq:3} and Assumption 2 hold. Then 
\begin{subequations}\label{eq:bcondm1}
\begin{align}
	{y'_{{\text{app}},n}}\left( 0 \right) &= y'\left( 0 \right) = \alpha_1,\label{eq:bcond0}\\
	{y'_{{\text{app}},n}}\left( b \right) &= y'^{(\alpha )}_{b,n,0} = \frac{\delta}{\gamma},\label{eq:bcond1}
\end{align}
\end{subequations}
and the solution derivative error values, $y'^{(\alpha )}_{b,n,i} - {y'_{\text{app},n}}\left( {x_{b,n,i}^{(\alpha )}} \right),\; i=1,\ldots,n$, are bounded by the following inequalities
\begin{equation}\label{eq:Sol20162}
\left| {y'^{(\alpha )}_{b,n,i} - {y'_{{\text{app}},n}}\left( {x_{b,n,i}^{(\alpha )}} \right)} \right| \le \frac{{{2^{ - 2n - 1}}A{\mkern 1mu} {b^{n + 1}}{\mkern 1mu} x_{b,n,i}^{(\alpha )}{\mkern 1mu} \Gamma (n + 2\alpha  + 1)\Gamma (\alpha  + 1)}}{{(n + 1)!\Gamma (n + \alpha  + 1)\Gamma (2\alpha  + 1)}}{\left\| {q_{b,n}^{(\alpha )}} \right\|_{\infty ,[0,b]}},
\end{equation}
where $y'^{(\alpha)}_{b,n,k} = y'\left(x_{b,n,k}^{(\alpha)}\right)$, for all $k$. Moreover, there exist some positive constants $B^{(\alpha)}$ and $C^{(\alpha)}$ dependent on $\alpha$ and independent of n such that the asymptotic solution derivative error values are bounded by
\begin{align}
\left| {y'^{(\alpha )}_{b,n,i} - {y'_{{\text{app}},n}}\left( {x_{b,n,i}^{(\alpha )}} \right)} \right| &\leq {B^{(\alpha )}}{b^{n + 1}}{\left(\frac{e}{4}\right)^{n}}\frac{{x_{b,n,i}^{(\alpha )}}}{{{n^{n + {\textstyle{3 \over 2}} - \alpha }}}},\quad i = 1,\ldots, n \wedge \alpha \ge 0,\label{eq:ineqasdev1}\\
\left| {y'^{(\alpha )}_{b,n,i} - {y'_{{\text{app}},n}}\left( {x_{b,n,i}^{(\alpha )}} \right)} \right| &\mathop  < \limits_ \sim
{C^{(\alpha )}}{2^{ - 2n - 1}}{e^n}{n^{ - n - \frac{3}{2}}}x_{b,n,i}^{(\alpha )}\,{b^{n + 1}},\quad i = 1,\ldots, n \wedge -1/2 < \alpha  < 0,\label{eq:ineqasdev2}
\end{align}
when $n \rightarrow \infty$. 
\end{thm}
\begin{proof}
Eqs. \eqref{eq:bcondm1} result by imposing the boundary conditions \eqref{eq:2} and \eqref{eq:3}. Inequalities \eqref{eq:Sol20162}, \eqref{eq:ineqasdev1}, and \eqref{eq:ineqasdev2} are established using Eq. \eqref{eq:48} and Theorem \ref{th:7}.
\end{proof}
\noindent Theorem \ref{th:10} shows that 
\begin{equation}
	\mathop {\lim }\limits_{n \to \infty } {\left\| {y' - {y'_{\text{app},n}}} \right\|_{\infty ,(0,b]}} = 0,
\end{equation}
with an error upper bound given by
\begin{equation}\label{eq:erruppb2}
	{\left\| {y' - {y'_{\text{app},n}}} \right\|_{\infty ,(0,b]}} = O\left( {\left\{ \begin{array}{l}
{\left( {\frac{e}{4}} \right)^n}{b^{n + 1}}{n^{ - n - \frac{3}{2} + \alpha }},\quad \alpha  \ge 0,\\
{2^{ - 2n - 1}}{e^n}\,{b^{n + 1}}{n^{ - n - \frac{3}{2}}},\quad  - 1/2 < \alpha  < 0
\end{array} \right.} \right).
\end{equation}
The following corollary highlights the uniform convergence of the solution function approximations $y_{\text{app},n}, n = 1, 2, \ldots$ under Assumption 2.
\begin{cor}[Uniform convergence]\label{cor:101}
Assume that $\left\|y^{(n+3)}\right\|_{\infty,(0,b]} < \infty$, and $y$ is approximated by the nodal truncated series expansion of shifted Gegenbauer basis polynomials \eqref{eq:57}, where the second-order derivative values $y''\left( {x_{b,n,i}^{(\alpha )}} \right), i = 0, \ldots ,n,$ are obtained by collocating the integral nonlinear Lane-Emden equation \eqref{eq:50} at the SFGGR points. Suppose also that the boundary conditions \eqref{eq:2} and \eqref{eq:3} and Assumption 2 hold. Then the sequence of solution function approximations  $\{y_{\text{app},n}\}_{n=1}^{\infty}$ converges uniformly on $(0,b]$ to the solution function $y$ with exponential rate.
\end{cor}
\begin{proof}
Eq. \eqref{eq:bcond1} implies that ${\lim _{n \to \infty }}{y_{\text{app},n}}(b)$ exists. We can also infer from inequalities \eqref{eq:ineqasdev1} and \eqref{eq:ineqasdev2} that there exists a positive integer $N$ such that
\begin{align}
{\left\| {y' - {y'_{\text{app},n}}} \right\|_{\infty ,(0,b]}} &< {B^{(\alpha )}}{b^{n + 2}}{\left(\frac{e}{4}\right)^{n}}\frac{1}{{{n^{n + {\textstyle{3 \over 2}} - \alpha }}}},\quad \alpha \ge 0,\label{eq:ineqasdev11}\\
{\left\| {y' - {y'_{\text{app},n}}} \right\|_{\infty ,(0,b]}} &\mathop  < \limits_ \sim
{C^{(\alpha )}}{2^{ - 2n - 1}}{e^n}{n^{ - n - \frac{3}{2}}}\,{b^{n + 2}},\quad -1/2 < \alpha  < 0,\label{eq:ineqasdev22}
\end{align}
for all $n \ge N$ and $x \in (0,b]$. Hence, the SGIPSM generates a sequence of derivative approximations $\{y'_{\text{app},n}\}_{n=1}^{\infty}$ that converges uniformly to the solution derivative function $y'$, for all $x \in (0,b]$ with exponential rate. This completes the required proof.
\end{proof}

Now, let us define the residual function $\mathcal{R}$ by 
	\begin{equation}\label{eq:70}
	\mathcal{R}(x) = {y''_{{\text{app},n}}}(x) + \frac{\alpha_{2}}{x}\, y'_{\text{app},n}(x) + f\left(x,y_{\text{app},n}(x)\right)\quad \forall x \in (0,b].
	\end{equation}
	Using Eqs. \eqref{eq:47}-\eqref{eq:49}, we have
	\begin{align}
	{{y''_{{\text{app},n}}}}\left(\bm{x}_{b,n}^{(\alpha)}\right) &= \bm{\Phi},\label{eq:71}\\
	y'_{\text{app},n}\left(\bm{x}_{b,n}^{(\alpha)}\right) &\approx \alpha_{1}\bm{1}_{n+1}+\mathbf{Q}_{b}^{(1)} \bm{\Phi},\label{eq:72}\\
	y_{\text{app},n}\left(\bm{x}_{b,n}^{(\alpha)}\right) &\approx \bar{\bm{x}}_{b,n}^{(\alpha)}+ {\mathbf{\Theta} _{b,n}} \bm{\Phi},\label{eq:73}
	\end{align}
	Hence, the residual function can be estimated at the SFGGR points by
	\begin{equation}\label{eq:74}
	\mathcal{R}\left( {{\bm{x}}_{b,n}^{(\alpha )}} \right) \approx {\mathbf{H}}_{b,n}^{(\alpha )}\,{\bm{\Phi }} + f\left( {{\bm{x}}_{b,n}^{(\alpha )},\bar{\bm{x}}_{b,n}^{(\alpha)} + {{\mathbf{\Theta }}_{b,n}}{\bm{\Phi }}} \right) + {\alpha _1}{\alpha _2}\,{\left( {{\bm{x}}_{b,n}^{(\alpha )}} \right)^ \div }.
	\end{equation}
We can further estimate the residual error bounds of the SGIPSM for Problem 1 at the SFGGR points through the following two useful theorems.
\begin{thm}\label{th:9new}
Assume that $\left\|y^{(n+k+2)}\right\|_{\infty,(0,b]} = A_{k} \in \mathbb{R}^{+}, \;k=0,1$, and suppose also that the function $f(x,y)$ satisfies Lipschitz condition in the variable $y$ on the set $\mathbb{S}_{b,n}^{(\alpha)} \times \mathbb{R}$ with Lipschitz constant $\lambda >0$. If the boundary conditions \eqref{eq:2} and \eqref{eq:3} and Assumption 1 hold, then the SGIPSM discretizes the integral nonlinear Lane-Emden equation \eqref{eq:50} at the SFGGR nodes $x_{b,n,i}^{(\alpha)} \in  \mathbb{S}_{b,n}^{(\alpha)},\; i=0,\ldots,n$, such that the residual function values $\mathcal{R}\left(x_{b,n,i}^{(\alpha)}\right),\; i=0,\ldots,n$ are bounded by the inequalities    
\begin{equation}
\begin{split}\label{eq:75}
\left|\mathcal{R}\left(x_{b,n,i}^{(\alpha)}\right)\right| &\leq \frac{{{2^{ - 2n - 1}}{{\left\| {q_{b,n}^{(\alpha )}} \right\|}_{\infty ,(0,b]}}\Gamma \left( {\alpha  + 1} \right)\,{b^{n + 1}}\,\Gamma \left( {n + 2\alpha  + 1} \right)}}{{(n + 1)!\,\Gamma \left( {2\alpha  + 1} \right)\,\,\Gamma \left( {n + \alpha  + 1} \right)}} \times \\
& \quad \left( {{A_1}\left( {b\,\lambda \left( {x_{b,n,i}^{(\alpha )} + \left| {\frac{\gamma }{\beta }} \right| + b} \right) + \left| {{\alpha _2}} \right|} \right) + {A_0}\,\lambda \,\left( {n + 1} \right)\left( {x_{b,n,i}^{(\alpha )} + b} \right)} \right),\quad i=0,\ldots,n.
\end{split}
\end{equation}
Moreover, when $n \rightarrow \infty$, the asymptotic residual function values are bounded by
\begin{equation}
\left| {\mathcal{R}\left( {x_{b,n,i}^{(\alpha )}} \right)} \right| \le {\cal L}_{1,\lambda,b}^{(\alpha )}{\left( {\frac{e}{4}} \right)^n}{b^{n + 1}}\left\{ \begin{array}{l}
{n^{ - n - \frac{3}{2} + \alpha }},\quad {A_0} = 0,\\
{n^{ - n - \frac{1}{2} + \alpha }},\quad {A_0} \ne 0,
\end{array} \right.\quad i = 0, \ldots ,n,
\label{eq:76}
\end{equation}
for $\alpha \geq 0$, and
\begin{equation}
\left| {\mathcal{R}\left( {x_{b,n,i}^{(\alpha )}} \right)} \right| \le {\cal L}_{2,\lambda,b}^{(\alpha )}{\left( {\frac{e}{4}} \right)^n}{b^{n + 1}}\left\{ \begin{array}{l}
{n^{ - n - \frac{3}{2}}},\quad {A_0} = 0,\\
{n^{ - n - \frac{1}{2}}},\quad {A_0} \ne 0,
\end{array} \right.\quad i = 0, \ldots ,n,
\label{eq:77}
\end{equation}
for $\frac{-1}{2}<\alpha < 0$, where $\mathcal{L}_{1,\lambda,b}^{(\alpha)}$ and $\mathcal{L}_{2,\lambda,b}^{(\alpha)}$ are positive constants dependent on $\lambda,b,\alpha$, and independent of $n$.
\end{thm}
\begin{proof}
Subtracting Eq. \eqref{eq:1} from Eq. \eqref{eq:70}, and using Lipschitz condition together with Eqs. \eqref{eq:3}, \eqref{eq:49}, \eqref{eq:51}, Theorems \ref{th:7} and \ref{th:8} yield
\begin{align}
	\left|\mathcal{R}\left(x_{b,n,i}^{(\alpha)}\right)\right| &\leq \frac{{|{\alpha _2}|}}{{x_{b,n,i}^{(\alpha )}}}\left| {_{y''}E_{b,n,i}^{(1),\alpha }} \right| + \lambda \left| {y_{b,n,i}^{(\alpha )} - {y_{{\text{app}},n}}\left( {x_{b,n,i}^{(\alpha )}} \right)} \right|\quad \forall i\label{eq:78new}\\
	&\le \frac{{|{\alpha _2}|}}{{x_{b,n,i}^{(\alpha )}}}\left| {_{y''}E_{b,n,i}^{(1),\alpha }} \right| + \lambda \left( {\left| {{{\frac{\gamma }{\beta }}_{y''}}E_{b,n,0}^{(1),\alpha }} \right| + \left| {_{y''}E_{b,n,0}^{(2),\alpha }} \right| + \left| {_{y''}E_{b,n,i}^{(2),\alpha }} \right|} \right)\quad \forall i,\label{eq:78}
	\end{align}
from which the proof is established.
	\end{proof}
\begin{thm}\label{th:9newk1}
Assume that $\left\|y^{(n+3)}\right\|_{\infty,(0,b]} < \infty$, and suppose also that the function $f(x,y)$ satisfies Lipschitz condition in the variable $y$ on the set $\mathbb{S}_{b,n}^{(\alpha)} \times \mathbb{R}$ with Lipschitz constant $\lambda >0$. If the boundary conditions \eqref{eq:2} and \eqref{eq:3} and Assumption 2 hold, then there exists a relatively small positive number $0 < \varepsilon \ll 1$ and a positive integer $N$ such that the SGIPSM discretizes the integral nonlinear Lane-Emden equation \eqref{eq:50} at the SFGGR nodes $x_{b,n,i}^{(\alpha)} \in  \mathbb{S}_{b,n}^{(\alpha)},\; i=0,\ldots,n$ with asymptotic residual function values $\mathcal{R}\left(x_{b,n,i}^{(\alpha)}\right),\; i=0,\ldots,n$ bounded by the inequalities    
\begin{equation}
\left| {{\cal R}\left( {x_{b,n,i}^{(\alpha )}} \right)} \right| \le \left\{ \begin{array}{l}
\varepsilon {\kern 1pt} \lambda  + {B^{(\alpha )}}{\kern 1pt} {b^{n + 1}}{\left( {\frac{e}{4}} \right)^n}{n^{ - n - \frac{3}{2} + \alpha }},\quad i = 0, \ldots ,n\, \wedge \,\alpha  \ge 0,\\
\varepsilon {\kern 1pt} \lambda  + {C^{(\alpha )}}{2^{ - 2n - 1}}{b^{n + 1}}{e^n}{n^{ - n - \frac{3}{2}}},\quad i = 0, \ldots ,n\, \wedge \, - 1/2 < \alpha  < 0,
\end{array} \right.
\end{equation}
for all $n \ge N$, where ${B^{(\alpha )}}$ and ${C^{(\alpha )}}$ are positive constants dependent on $\alpha$, and independent of $n$.
\end{thm}
\begin{proof}
By the uniform convergence of Corollary \ref{cor:101}, we find that
\[\left| {y_{b,n,i}^{(\alpha )} - {y_{{\text{app}},n}}\left( {x_{b,n,i}^{(\alpha )}} \right)} \right| \le \varepsilon\quad \forall n \ge N.\]
The proof is established by combining this result with inequality \eqref{eq:78new}.
\end{proof}
	Similarly, to furnish the exponential convergence rate of the SGIPSM for Problem 2, notice that the residual term in this case takes the following special form:
	\begin{equation}\label{eq:79}
	\mathcal{R}(x) = L(y_{\text{app},n}(x)) - g(x)\quad \forall x \in (0,b],
	\end{equation}
	where the differential operator $L = \displaystyle{\frac{d^{2}}{dx^{2}} + \frac{\alpha_{2}}{x} \frac{d}{dx}} + p(x)$. The following two corollaries are direct results of Theorems \ref{th:9} and \ref{th:9newk1}.
\begin{cor}\label{cor:3}
Assume that $\left\| y^{(n+k+2)}\right\|_{\infty,(0,b]}=A_{k} \in \mathbb{R}^{+},\; k=0,1$, and $\left\|p\right\|_{\infty,(0,b]} = M \in \mathbb{R}_0^+$. If the boundary conditions \eqref{eq:2} and \eqref{eq:3} and Assumption 1 hold, then the SGIPSM discretizes the integral nonlinear Lane-Emden equation \eqref{eq:50} at the SFGGR nodes $x_{b,n,i}^{(\alpha)} \in \mathbb{S}_{b,n}^{(\alpha)},\; i=0,\ldots,n$, such that the residual function values $\mathcal{R}\left(x_{b,n,i}^{(\alpha)}\right),\; i=0,\ldots,n$ are bounded by the inequalities 
\begin{equation}\label{eq:80}
\begin{split}
\left|\mathcal{R}\left(x_{b,n,i}^{(\alpha)}\right)\right| &\leq \frac{{{2^{ - 2n - 1}}{{\left\| {q_{b,n}^{(\alpha )}} \right\|}_{\infty ,(0,b]}}\Gamma \left( {\alpha  + 1} \right)\,{b^{n + 1}}\,\Gamma \left( {n + 2\alpha  + 1} \right)}}{{(n + 1)!\,\Gamma \left( {2\alpha  + 1} \right)\,\,\Gamma \left( {n + \alpha  + 1} \right)}} \times \\
& \quad \left( {{A_1}\left( {b\,M \left( {x_{b,n,i}^{(\alpha )} + \left| {\frac{\gamma }{\beta }} \right| + b} \right) + \left| {{\alpha _2}} \right|} \right) + {A_0}\,M \,\left( {n + 1} \right)\left( {x_{b,n,i}^{(\alpha )} + b} \right)} \right),\quad i=0,\ldots,n.
\end{split}
\end{equation}
Moreover, when $n \rightarrow \infty$, the asymptotic residual function values are bounded by
\begin{equation}\label{eq:81}
\left| {\mathcal{R}\left( {x_{b,n,i}^{(\alpha )}} \right)} \right| \le {\cal L}_{1,M,b}^{(\alpha )}{\left( {\frac{e}{4}} \right)^n}{b^{n + 1}}\left\{ \begin{array}{l}
{n^{ - n - \frac{3}{2} + \alpha }},\quad {A_0} = 0,\\
{n^{ - n - \frac{1}{2} + \alpha }},\quad {A_0} \ne 0,
\end{array} \right.\quad i = 0, \ldots ,n,
\end{equation}
for $\alpha \geq 0$, and
\begin{equation}
\left| {\mathcal{R}\left( {x_{b,n,i}^{(\alpha )}} \right)} \right| \le {\cal L}_{2,M,b}^{(\alpha )}{\left( {\frac{e}{4}} \right)^n}{b^{n + 1}}\left\{ \begin{array}{l}
{n^{ - n - \frac{3}{2}}},\quad {A_0} = 0,\\
{n^{ - n - \frac{1}{2}}},\quad {A_0} \ne 0,
\end{array} \right.\quad i = 0, \ldots ,n,
\label{eq:82}
\end{equation}
for $\frac{-1}{2}<\alpha < 0$, where $\mathcal{L}_{1,M,b}^{(\alpha)}$ and $\mathcal{L}_{2,M,b}^{(\alpha)}$ are positive constants dependent on $M,b,\alpha$, and independent of $n$.
\end{cor}
\begin{cor}\label{cor:4}
Assume that  $\left\|p\right\|_{\infty,(0,b]} = M \in \mathbb{R}_0^+$. If the boundary conditions \eqref{eq:2} and \eqref{eq:3} and Assumption 2 hold, then there exists a relatively small positive number $0 < \varepsilon \ll 1$ and a positive integer $N$ such that the SGIPSM discretizes the integral nonlinear Lane-Emden equation \eqref{eq:50} at the SFGGR nodes $x_{b,n,i}^{(\alpha)} \in  \mathbb{S}_{b,n}^{(\alpha)},\; i=0,\ldots,n$ with asymptotic residual function values $\mathcal{R}\left(x_{b,n,i}^{(\alpha)}\right),\; i=0,\ldots,n$ bounded by the inequalities    
\begin{equation}
\left| {{\cal R}\left( {x_{b,n,i}^{(\alpha )}} \right)} \right| \le \left\{ \begin{array}{l}
\varepsilon {\kern 1pt} M + {B^{(\alpha )}}{\kern 1pt} {b^{n + 1}}{\left( {\frac{e}{4}} \right)^n}{n^{ - n - \frac{3}{2} + \alpha }},\quad i = 0, \ldots ,n\, \wedge \,\alpha  \ge 0,\\
\varepsilon {\kern 1pt} M  + {C^{(\alpha )}}{2^{ - 2n - 1}}{b^{n + 1}}{e^n}{n^{ - n - \frac{3}{2}}},\quad i = 0, \ldots ,n\, \wedge \, - 1/2 < \alpha  < 0,
\end{array} \right.
\end{equation}
for all $n \ge N$, where ${B^{(\alpha )}}$ and ${C^{(\alpha )}}$ are positive constants dependent on $\alpha$, and independent of $n$.
\end{cor}
\begin{proof}
The proof is similar to that of Theorem \ref{th:9newk1}.
\end{proof}
\section{Numerical experiments}
\label{sec:NE1}
In this section, we apply the proposed SGIPSM on five test examples of Problems 1 and 2. All numerical experiments were carried out using MATLAB R2011b (7.13.0.564) software installed on a personal laptop equipped with an Intel(R) Core(TM) i3-4500U CPU with 1.70 GHz speed running on Windows 7 Ultimate 32-bit operating system. The exact solutions were calculated using MATHEMATICA 7 with $18$ digits of precision maintained in internal computations. In all numerical tests, by AE and RE, we mean the absolute and relative errors, respectively. MAE and AE$_b$ denote the maximum AE and the AE in the computation of the solution function $y$ at $x = b$, respectively. Furthermore, by $\kappa_{\infty}$, we mean the infinity norm condition number of the resulting linear algebraic system associated with Problem 2. The resulting linear algebraic systems of equations were solved using MATLAB ``mldivide'' algorithm. Nonlinear algebraic systems associated with Problem 1 were solved by MATLAB ``fsolve'' solver with termination tolerance ``TolFun'' set at $10^{-15}$.\\
 
\noindent \textbf{Example 1.} Consider Problem 2 with $\alpha_{1}=0, \alpha_{2}=1, \beta=1, \gamma=0, \delta=0, b=1, p(x)=0$, and $g(x)=\left(8/\left(8-x^2\right)\right)^{2} $. The exact solution is $y(x)=2\log\left(7/\left(8-x^2\right)\right)$. Figure \ref{fig1} shows the plots of the exact solution, approximate solution, AE, and MAE on $[0,1]$; in addition, the figure shows $\kappa_{\infty}$ obtained by the SGIPSM for several values of $n$ and $\alpha$. 

Example 1 was previously solved by \cite{caglar2006b} and later by \cite{yuzbacsi2013improved} using a B-spline method and an improved Bessel's method, respectively. 
Table \ref{table1.2} shows the power of the state-of-the-art SGIPSM when compared with the B-spline method of \cite{caglar2006b}, as illustrated by gaining extra digits of precision using far lesser number of collocation points. The SGIPSM outperforms also the method of \cite{yuzbacsi2013improved} as clearly observed from Table \ref{table1.1}, where higher-order approximations are achieved using the same number of collocation points. Moreover, the SGIPSM performs a single discretization to the integral form of the problem against two discretizations for the method of \cite{yuzbacsi2013improved} applied to the problem and its associated error problem; thus, significantly reducing the necessary computational cost to establish the same level of accuracy.  In addition, thanks to the discretization process employing the SFGGR collocation points, the boundary condition, $y(1) = 0$, is satisfied exactly by the SGIPSM.  
\begin{figure}[H]
\centering
\includegraphics[scale=0.6]{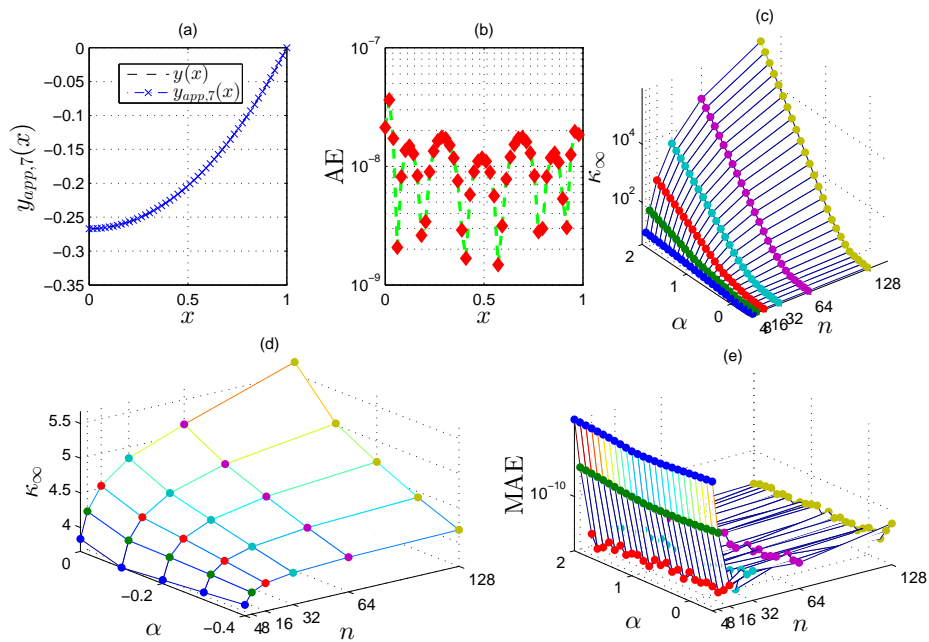}
\caption{The numerical simulation of Example $1$ using the SGIPSM. Figure (a) shows the exact and approximate solution on $[0,1]$ obtained using $n = 7$ and $\alpha=1.1$. Figure (b) shows the corresponding AE in log-lin scale. Both figures were generated using 50 linearly spaced nodes from 0 to 1. Figures (c) and (e) show $\kappa_{\infty}$ and the MAE for $n = 2^{k}, k = 2(1)7$ and $\alpha = -0.4(0.1)2$ with linear and logarithmic scales on the $z$ axis, respectively. Figure (d) shows further $\kappa_{\infty}$ for $\alpha = -0.4(0.1)0$ with logarithmic scale on the $z$ axis.}
\label{fig1}
\end{figure}
\begin{table}[!ht]
\centering
\caption{The REs of the B-spline method \cite{caglar2006b} and the SGIPSM.} 
\begin{tabular}{c c c}
\toprule
$x$ & \textbf{B-spline method} \cite{caglar2006b}  & \textbf{SGIPSM}\\
& $n=20$ & $n=5, \alpha=0.1$ \\
\midrule
  0    & 9.8160e-005  & 1.8405e-006 \\      
 0.05  & 9.8756e-005  & 1.8512e-006 \\
 0.1   & 1.0122e-004  & 3.2213e-006 \\
 0.2   & 1.0231e-004  & 9.0436e-006 \\      
 0.3   & 1.0119e-004  & 3.6570e-006 \\    
 0.4   & 1.0425e-004  & 7.1760e-006 \\
 0.5   & 1.0616e-004  & 1.3045e-005 \\
 0.6   & 1.0911e-004  & 6.5036e-006 \\
 0.7   & 1.0925e-004  & 1.1139e-005 \\      
 0.8   & 1.1398e-004  & 2.5024e-005 \\    
 0.9   & 1.1117e-004  & 2.4411e-006 \\
\midrule
AE$_1$ &  0  &  0\\
\bottomrule
\end{tabular}
\label{table1.2}   
 \end{table}
\begin{table}[!ht]
\centering
\caption{The REs of Bessel's method \cite{yuzbacsi2013improved} and the SGIPSM using eight collocation points.} 
\begin{tabular}{c c c}
\toprule
$x$ & \textbf{Bessel's method} \cite{yuzbacsi2013improved}  & \textbf{SGIPSM}\\
& & $\alpha=1.1$\\
\midrule
  0  & 5.9957e-006   & 7.9499e-008\\      
 0.2 & 6.0946e-006   & 8.6474e-009\\    
 0.4 & 6.8724e-006   & 5.4461e-010\\
 0.6 & 8.8719e-006   & 2.8767e-008\\
 0.8 & 1.5473e-005   & 4.0944e-008\\
\midrule
AE$_1$ &  1.6198e-016  &      0\\
\bottomrule
\end{tabular}
\label{table1.1}   
\end{table}
\noindent \textbf{Example 2.} Consider Problem 1 with $\alpha_{1}=0, \alpha_{2}=2, \beta=1, \gamma=0, \delta=\sqrt{3}/2, b=1$, and $f(x,y)=y^{5}$. The exact solution is $y(x)=1/\sqrt{1+x^{2}/{3}}$. Figure \ref{fig2} shows the plots of the exact solution, approximate solution, AE, and MAE on $[0,1]$ for several values of $n$ and $\alpha$. 
\cite{kanth2007cubic} solved the problem using a cubic spline method combined with a quesilinearization technique that reduces the nonlinear problem into a sequence of linear problems. \cite{turkyilmazoglu2013effective} later approached the solution of the problem using a continuous functions series expansion technique in which the problem is discretized into a system of nonlinear algebraic equations by expanding the solution 
in a truncated Maclaurin series. The expansion coefficients are then calculated through a Galerkin-like method. Table \ref{table2.1} shows the rapid convergence rates and the memory minimizing feature of the SGIPSM against the cubic spline method of \cite{kanth2007cubic}, where the former method produces higher accuracy using relatively much smaller number of collocation points. On the other hand, the latter method requires the solution of a nonlinear system of equations of order $53$, versus only $8$ for the SGIPSM to finally yield approximations of lower accuracy. The SGIPSM gives also greater accuracy than that of \cite{turkyilmazoglu2013effective} for several values of $n$ as clearly observed from Table \ref{table2.2}. The small approximation error at $x = 1$ is due to the inexactness in the computer representation of the real number $\sqrt{3}/2$.


\begin{figure}[H]
\centering
\includegraphics[width=11cm, height=11cm]{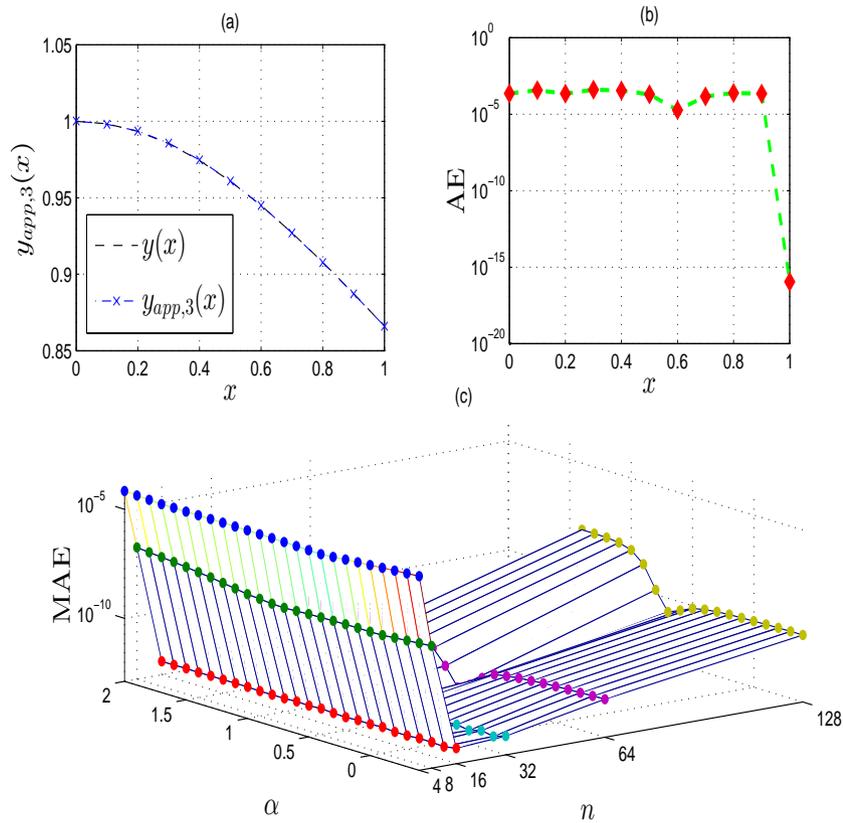}
\caption{The numerical simulation of Example $2$ using the SGIPSM. Figure (a) shows the exact and approximate solution on $[0,1]$ obtained using $n = 3$ and $\alpha=0.8$. Figure (b) shows the corresponding AE in log-lin scale. Both figures were generated using 11 linearly spaced nodes from 0 to 1. Figure (c) shows the MAE for $n = 2^{k}, k = 2(1)7$ and $\alpha = -0.4(0.1)2$ with logarithmic scale on the $z$ axis.}                
\label{fig2}
\end{figure}
\begin{table}[!ht]
\centering
\caption{The REs of the Cubic spline method \cite{kanth2007cubic} and the SGIPSM} 
\begin{tabular}{c c c}
\toprule
$x$ & \textbf{Cubic spline method} \cite{kanth2007cubic}  & \textbf{SGIPSM}  \\
&$n=50$ & {$n=7, \alpha=1.4$} \\
\midrule
  0  & 2.2341e-006 & 1.1021e-007 \\      
 0.1 & 2.1616e-006 & 3.7718e-009 \\    
 0.2 & 1.9619e-006 & 1.7350e-008 \\
 0.3 & 1.6583e-006 & 8.4623e-008 \\
 0.4 & 1.2894e-006 & 3.6471e-008 \\
 0.5 & 9.0037e-007 & 1.6719e-008 \\
 0.6 & 5.3675e-007 & 3.4350e-008 \\
 0.7 & 2.3891e-007 & 7.9305e-008 \\
 0.8 & 3.8253e-008 & 4.9866e-009 \\
 0.9 & 4.4620e-008 & 1.9892e-009 \\
 \midrule
AE$_1$ &  1.3323e-015 &    4.4409e-016 \\
\bottomrule
\end{tabular}
\label{table2.1}    
 \end{table}
\begin{table}[!ht]
\centering
\caption{The MAEs of \cite{turkyilmazoglu2013effective}'s method and the SGIPSM.} 
\begin{tabular}{c c c}
\toprule
$n$ & \textbf{\cite{turkyilmazoglu2013effective}'s method} & \textbf{SGIPSM}\\
& & {MAE/($\alpha$)} \\
\midrule
  3  & 1.1100e-003 & 3.9711e-004/($\alpha=0.8$)   \\      
  6  & 5.5622e-006 & 1.7118e-006/($\alpha=-0.1$)  \\    
  8  & 5.2440e-008 & 2.6347e-008/($\alpha=0.8$)   \\
\midrule
\end{tabular}
\label{table2.2}   
 \end{table}

\noindent \textbf{Example 3.} Consider Problem 2 with $\alpha_{1}=0, \alpha_{2}=2, \beta=1, \gamma=0, \delta=5.5, b=1, p(x)=-4$, and $g(x)=-2$. The exact solution is $y(x)=0.5 + 5 \sinh(2x)/(x \sinh(2))$. Figure \ref{fig3} shows the plots of the exact solution, approximate solution, AE, MAE, and $\kappa_{\infty}$ obtained by the SGIPSM for several values of $n$ and $\alpha$. \cite{kanth2005cubic} solved the problem by modifying the original differential equation \eqref{eq:4} at the singular point $x=0$ via L'H\^{o}pital's rule, and then collocating the modified differential equation at equispaced grid points using a cubic spline method. Such a numerical scheme discretized the problem into a tridiagonal system of equations that was solved by Thomas algorithm. Table \ref{table3.1} shows the preference of the proposed SGIPSM against the method of \cite{kanth2005cubic} in terms of higher accuracy and reduced computational cost.


\begin{figure}[H]
\centering
\includegraphics[width=11cm, height=11cm]{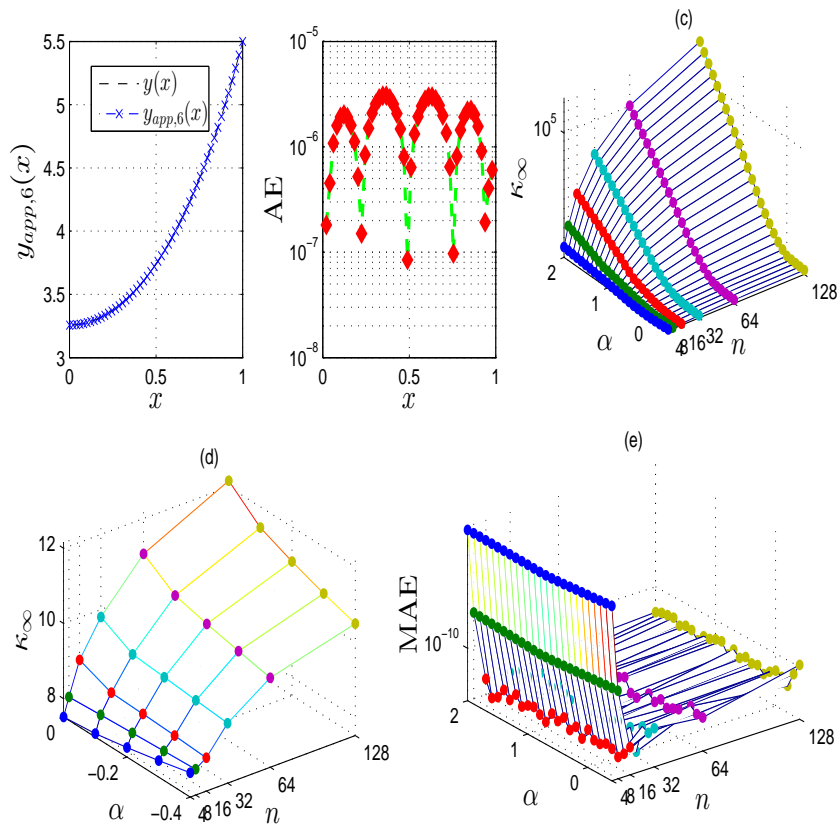}
\caption{The numerical simulation of Example $3$ using the SGIPSM. Figure (a) shows the exact and approximate solution on $[0,1]$ obtained using $n = 6$ and $\alpha=-0.2$. Figure (b) shows the corresponding AE in log-lin scale. Both figures were generated using 50 linearly spaced nodes from 0 to 1. Figure (c) shows $\kappa_{\infty}$ for $n = 2^{k}, k = 2(1)7$ and $\alpha = -0.4(0.1)2$. Figure (d) shows further $\kappa_{\infty}$ for $\alpha = -0.4(0.1)0$ with logarithmic scale on the $z$ axis. Figure (e) shows the MAE for $n = 2^{k}, k = 2(1)7$ and $\alpha = -0.4(0.1)2$ with logarithmic scale on the $z$ axis.}                                                                                                                                         
\label{fig3}
\end{figure}
  
 
  
\begin{table}[ht]
\centering
\caption{The REs of the cubic-spline method \cite{kanth2005cubic} and the SGIPSM.} 
\begin{tabular}{c c c}
\toprule
$x$ & \textbf{Cubic spline method} \cite{kanth2005cubic}  & \textbf{SGIPSM}\\
&{$n=20$}&{$n=6, \alpha=-0.2 $} \\
\midrule
 0.25  & -             & 1.4056e-008 \\
 0.05  & 8.9610e-005   & 2.2797e-007 \\
 0.075 & -             & 4.4184e-007 \\
 0.1   & 8.9087e-005   & 5.7829e-007 \\
 0.2   & 8.6627e-005   & 1.9466e-007 \\      
 0.3   & 8.2151e-005   & 7.0896e-007 \\    
 0.4   & 7.6256e-005   & 7.5662e-007 \\
 0.5   & 6.8275e-005   & 7.4427e-008 \\
 0.6   & 5.8753e-005   & 7.2706e-007 \\
 0.7   & 4.7165e-005   & 4.2395e-007 \\      
 0.8   & 3.3711e-005   & 3.3767e-007 \\    
 0.9   & 1.7861e-005   & 3.0010e-007 \\
 \midrule
AE$_1$ &  0  &      0      \\
\bottomrule
\end{tabular} 
\label{table3.1}   
 \end{table}
\noindent \textbf{Example 4.} Consider Problem 1 with $\alpha_{1}=0, \alpha_{2}=1, \beta=1, \gamma=0, \delta=2\ln\left(\left(4-2\sqrt{2}\right)/\left(7.75-4.5\sqrt{2}\right)\right), b=1.5$, and $f(x,y)=e^{y}$. The exact solution is $y(x)=2\ln\left((c+1) / \left(c x^{2}+1\right)\right)$ with $c=3-2\sqrt{2}$. Figure \ref{fig4} shows the plots of the exact solution, approximate solution, and AE on $[0,1.5]$ using $n = 5$ and $\alpha=0.9$, in addition to the MAE for $n = 2^{k}, k = 2(1)7$ and $\alpha = -0.4(0.1)2$. The REs of the SGIPSM using $n=5$ and $\alpha=0.9$ is shown in Table \ref{table4.1}. This problem was previously solved by \cite{khuri2010novel} using a decomposition method in combination with the cubic B-spline collocation technique. In particular, the former method is employed in the vicinity of the singularity $x = 0$, where the nonlinear term $f(x,y)$ is decomposed in terms of the Adomian polynomials, while the latter method is implemented outside this range. Table \ref{table4.2} shows however higher-order approximations in favor of the SGIPSM using several values of $n$.  
\begin{figure}[H]
\centering
\includegraphics[width=10cm, height=10cm]{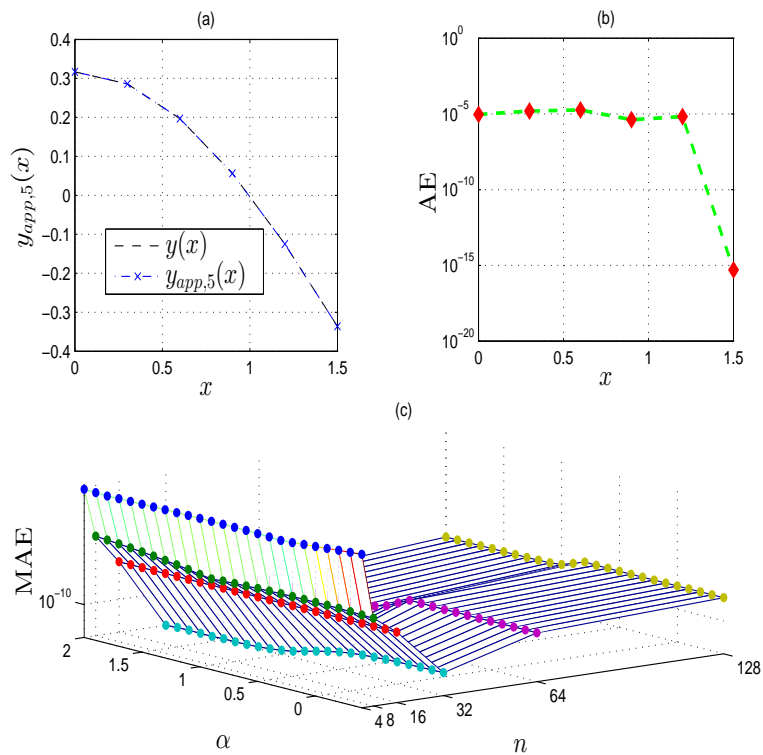}
\caption{The numerical simulation of Example $4$ using the SGIPSM. Figure (a) shows the exact and approximate solution on $[0,1.5]$ obtained using $n = 5$ and $\alpha=0.9$. Figure (b) shows the corresponding AE in log-lin scale. Both figures were generated using 6 linearly spaced nodes from 0 to 1.5. Figure (c) shows the MAE for $n = 2^{k}, k = 2(1)7$ and $\alpha = -0.4(0.1)2$ with logarithmic scale on the $z$ axis.}
\label{fig4}                                                                                      
\end{figure}


\begin{table}[!ht]
\centering
\caption{The REs of the SGIPSM using $n=5$.} 
\begin{tabular}{c c}
\toprule
$x$ &  \textbf{SGIPSM}\\
&  $ \alpha=0.9$ \\
\midrule
0   & 2.9203e-005  \\      
0.3 & 5.3820e-005  \\    
0.6 & 9.1514e-005  \\
0.9 & 7.3363e-005 \\
1.2 & 5.2870e-005  \\
\midrule
AE$_{1.5}$& 4.996e-016 \\
\bottomrule
\end{tabular}
\label{table4.1}
\end{table}  
\begin{table}[!ht]
\centering
\caption{The MAEs of \cite{khuri2010novel}'s method and the SGIPSM using $n=5(5)20$.} 
\begin{tabular}{c c c}
\toprule
$n$ & \textbf{\cite{khuri2010novel}'s method} & \textbf{SGIPSM} \\
& &{MAE/($\alpha$)} \\
\midrule
  5 & 2.37e-005  &  1.8012e-005/(0.9)\\      
 10 & 6.18e-006  &  1.3886e-009/(0.5)\\    
 15 & 3.03e-006  &  1.1013e-013 /(-0.1)\\
 20 & 1.56e-006  &  5.8287e-015 /(2)  \\
\midrule
\end{tabular}
\label{table4.2}   
 \end{table}
\noindent \textbf{Example 5.} Consider Problem 1 with $\alpha_{1}=-1, \alpha_{2}=2, \beta=0, \gamma=1, \delta=-1, b=1$, and $f(x,y)=\sin(y)-\cos(x)+2/x$. The exact solution is $y(x)=\pi/2-x$. Figure \ref{fig5} shows the plots of the exact solution, approximate solution, AE, and MAE on $[0,1]$ for several values of $n$ and $\alpha$. Table \ref{table5.1} verifies the strength of the proposed method as shown by the exceedingly accurate approximations obtained by the method using as small as $8$ collocation points.
\begin{figure}[H]
\centering
\includegraphics[width=11cm, height=11cm]{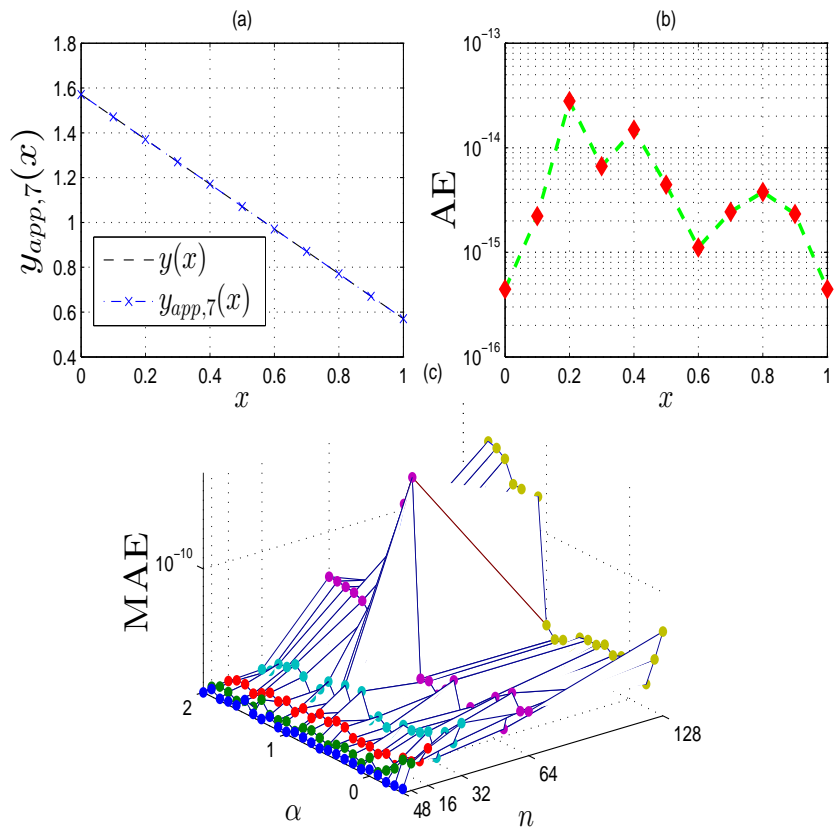}
\caption{The numerical simulation of Example $5$ using the SGIPSM. Figure (a) shows the exact and approximate solution on $[0,1]$ obtained using $n = 7$ and $\alpha=0.9$. Figure (b) shows the corresponding AE in log-lin scale. Both figures were generated using 11 linearly spaced nodes from 0 to 1. Figure (c) shows the MAE for $n = 2^{k}, k = 2(1)7$ and $\alpha = -0.4(0.1)2$ with logarithmic scale on the $z$ axis.}
\label{fig5}
\end{figure}
%
\begin{table}[!ht]
\centering
\caption{The REs of the SGIPSM using $n=7$.} 
\begin{tabular}{c c}
\toprule
$x$ &  \textbf{SGIPSM} \\
 & {$\alpha=0.9$} \\
\midrule
  0  & 2.8272e-016 \\      
 0.1 & 1.6607e-015 \\  
 0.2 & 2.0410e-014 \\
 0.3 & 5.2419e-015 \\
 0.4 & 1.2707e-014 \\
 0.5 & 4.1473e-015 \\
 0.6 & 1.1436e-015 \\
 0.7 & 2.8049e-015 \\
 0.8 & 5.0413e-015 \\  
 0.9 & 3.3102e-015 \\    
 \midrule
AE$_1$ & 4.4409e-016\\  
\bottomrule
\end{tabular}
 \label{table5.1}
 \end{table}  
%
\section{Conclusion and discussion}
\label{Conc}
In this work, we introduced the SGIPSM: a novel, highly accurate, exponentially convergent, and efficient numerical method for solving singular linear and nonlinear Lane-Emden equations provided with certain mixed Neumann and Robin boundary conditions. The proposed method exploits the well conditioning of numerical integral operators by working on the integral formulation of the problem. The reduced problem is then collocated at the SFGGR nodes, 
and the definite integrals are approximated using some novel SFGGR-based shifted Gegenbauer integration matrices. Such numerical integration operators are constant operators that can be stored for certain sets of SFGGR points, and invoked later when implementing the proposed computational algorithms. The numerical scheme eventually discretizes the problem into a system of linear/nonlinear equations that can be solved using standard linear/nonlinear system solvers. The exponential convergence of the SGIPSM is verified theoretically through Theorem \ref{th:9} and Corollary \ref{cor:101}. Moreover, estimates of the residual error bounds were provided by Theorems \ref{th:9new} and \ref{th:9newk1} and Corollaries \ref{cor:3} and \ref{cor:4}. Five test examples were solved to assess the power of the SGIPSM. In all test examples, the exponential convergence of the SGIPSM is clearly demonstrated, where higher-order approximations are achieved using relatively small numbers of collocation nodes. Moreover, the MAE of the method generally approaches the machine epsilon at either $n = 16$ or $n = 32$, and then degrades 
slightly for increasing values of $n$ due to the presence of round-off errors. Adopting the SFGGR nodes for collocating Problems 1 and 2 permits for approximating the solution function $y$ at $x = b$ very accurately. The numerical simulations of Examples 1 and 3 manifest that $\kappa_{\infty}$ increases monotonically for increasing values of $n$ with a swelling rate of increase for increasing positive values of $\alpha$. For instance, the value of $\kappa_{\infty}$ of Example 1 jumps from about $4$ at $\alpha = -0.4$ and $n = 128$ into more than $10^5$ at $\alpha = 2$ for the same value of $n$. A similar pattern is also observed in Example 3, where the value of $\kappa_{\infty}$ surges from about $10$ at $\alpha = -0.4$ and $n = 128$ into more than $10^6$ at $\alpha = 2$ for the same value of $n$. Therefore, for very large values of $n$, we would expect a significant drop in the precision of the approximate solutions for increasing positive values of $\alpha$-- a result that is reflected by the theoretical aftermath of Eqs. \eqref{eq:erruppb1} and \eqref{eq:erruppb2}. We draw the attention of the reader also to the significant observation pointed out by \cite{elgindy2013optimal} that the Gegenbauer quadratures employed in the discretization process `may become sensitive to round-off errors
for positive and large values of the parameter $\alpha$ due to the narrowing effect of the Gegenbauer weight function,' which drives the quadratures to become more extrapolatory and possibly give rise to poor integral approximations. On the other hand, we can infer from both examples that the resulting linear algebraic systems remain well conditioned as $n$ tends to infinity for $\alpha \in [-0.4,0]$ with a general monotonic increase of $\kappa_{\infty}$ for increasing values of $n$ and $\alpha$. That is, Gegenbauer basis polynomials associated with negative values of $\alpha$ generally lead to more well conditioned linear systems than their rival Chebyshev and Legendre bases polynomials as $n$ tends to infinity. We observe though that $\kappa_{\infty}$ remains plausible for $\alpha \in [-0.4,1]$ in the range $n \le 16$. In fact, if we denote the interval $[-1/2 + \varepsilon, r]$ by ${I_{\varepsilon,r}^G}$, for some relatively small positive number $\varepsilon$ and $r \in [1,2]$, then we closely follow the rule of thumb recommended by \cite{elgindy2016a} in the sense of generally performing collocations for values of $\alpha \in {I_{\varepsilon,r}^G}$, for small/medium numbers of collocation points and shifted Gegenbauer expansion terms; however, collocations at the shifted flipped-Chebyshev-Gauss-Radau points should be put into effect for large numbers of collocation points and shifted Gegenbauer expansion terms if the approximations are sought in the infinity norm (Chebyshev norm). Numerical comparisons with other traditional methods shown in Tables \ref{table1.2}--\ref{table3.1}, and \ref{table4.2} prove the preference of the proposed method over traditional methods in the literature with respect to the accuracy of the calculated approximations, convergence rate, and the employed number of collocation points. The presented SGIPSM is easily programmed, and can be extended to solve various problems in many areas of science and applications.         
%
\appendix
\section{Computational algorithms}
\label{app:Alg1}
\begin{algorithm}
\renewcommand{\thealgorithm}{A.1}
\caption{The SGIPSM for solving Problem 1 under Assumption $1$}
\label{algorithm:1}
\begin{algorithmic}[1]
\State Input: Real numbers $\alpha_1, \alpha_2, \beta, \gamma, \delta$; a positive real number $b$; a real-valued function $f$; $\bm{x}_{b,n}^{(\alpha )} \in \mathbb{S}_{b,n}^{(\alpha )}$, for some positive integer $n$; ${\mathbf{Q}}_b^{(1)}; {\mathbf{Q}}_b^{(2)}$; solution points vector $\bm{z} = (z_i) \in \mathbb{R}^m: z_i \in (0, b], i = 0, \ldots, m$, for some $m \in \mathbb{Z}^{+}_0$.
\State Output: Approximate solution vector $y_{\text{app},n}\left(\bm{z}\right)$.\quad \Comment{$y_{\text{app},n}(0)$ can be computed using Eq. \eqref{eq:51}.}
\State Calculate $\mathbf{H}_{b,n}^{(\alpha)}, \bar{\bm{x}}_{b,n}^{(\alpha)}$, and $\mathbf{\Theta }_{b,n}$ using Eqs. \eqref{eq:53}-\eqref{eq:55}.
\State Solve the nonlinear system \eqref{eq:52} for $\bm{\Phi}$.
\State Calculate the approximate solution vector ${y_{\text{app},n}}\left( {{\bm{x}}_{b,n}^{(\alpha )}} \right)$ using Eq. \eqref{eq:56}.
\State Calculate ${y_{\text{app},n}}\left( \bm{z} \right)$ using Eq. \eqref{eq:57}; Stop.
\end{algorithmic}
\end{algorithm}
\begin{algorithm}
\renewcommand{\thealgorithm}{A.2}
\caption{The SGIPSM for solving Problem 2 under Assumption $1$}
\label{algorithm:2}
\begin{algorithmic}[1]
\State Input: Real numbers $\alpha_1, \alpha_2, \beta, \gamma, \delta$; a positive real number $b$; real-valued functions $p, g$; $\bm{x}_{b,n}^{(\alpha )} \in \mathbb{S}_{b,n}^{(\alpha )}$, for some positive integer $n$; ${\mathbf{Q}}_b^{(1)}; {\mathbf{Q}}_b^{(2)}$; solution points vector $\bm{z} = (z_i) \in \mathbb{R}^m: z_i \in (0, b], i = 0, \ldots, m$, for some $m \in \mathbb{Z}^{+}_0$.
\State Output: Approximate solution vector $y_{\text{app},n}\left(\bm{z}\right)$.\quad \Comment{$y_{\text{app},n}(0)$ can be computed using Eq. \eqref{eq:51}.}
\State Calculate $\mathbf{H}_{b,n}^{(\alpha)}, \bar{\bm{x}}_{b,n}^{(\alpha)}$, and $\mathbf{\Theta }_{b,n}$ using Eqs. \eqref{eq:53}-\eqref{eq:55}.
\State Calculate the coefficient matrix $\mathcal{A}$ and the constant vector $\mathcal{B}$ using Eqs. \eqref{eq:CoeffAB1} and \eqref{eq:CoeffAB2}, respectively.
\State Solve the linear algebraic system \eqref{eq:58} for $\bm{\Phi}$.
\State Calculate the approximate solution vector ${y_{\text{app},n}}\left( {{\bm{x}}_{b,n}^{(\alpha )}} \right)$ using Eq. \eqref{eq:56}.
\State Calculate ${y_{\text{app},n}}\left( \bm{z} \right)$ using Eq. \eqref{eq:57}; Stop.
\end{algorithmic}
\end{algorithm}
\begin{algorithm}
\renewcommand{\thealgorithm}{A.3}
\caption{The SGIPSM for solving Problem 1 under Assumption $2$}
\label{algorithm:3}
\begin{algorithmic}[1]
\State Input: Real numbers $\alpha_1, \alpha_2, \gamma, \delta$; a positive real number $b$; a real-valued function $f$; $\bm{x}_{b,n}^{(\alpha )} \in \mathbb{S}_{b,n}^{(\alpha )}$, for some positive integer $n$; ${\mathbf{Q}}_b^{(1)}; {\mathbf{Q}}_b^{(2)}$; solution points vector $\bm{z} = (z_i) \in \mathbb{R}^m: z_i \in (0, b], i = 0, \ldots, m$, for some $m \in \mathbb{Z}^{+}_0$.
\State Output: Approximate solution vector $y_{\text{app},n}\left(\bm{z}\right)$.
\State Calculate $\mathbf{H}_{b,n}^{(\alpha)}, {\mathbf{\hat H}}_{b,n}^{(\alpha )}, \hat {f}_{b,n}^{(\alpha )}(\bm{\Phi})$, and ${\bm{\hat x}}_{b,n}^{(\alpha )}$ using Eqs. \eqref{eq:53}, \eqref{eq:hatH1}-\eqref{hatx1}.
\State Solve the nonlinear system \eqref{eq:9new} for $\bm{\Psi} = \left[\bm{\Phi};y_{\text{app},n}(0)\right]$.
\State Calculate the approximate solution vector ${y_{\text{app},n}}\left( {{\bm{x}}_{b,n}^{(\alpha )}} \right)$ using Eq. \eqref{eq:14new}.
\State Calculate ${y_{\text{app},n}}\left( \bm{z} \right)$ using Eq. \eqref{eq:57}; Stop.
\end{algorithmic}
\end{algorithm}
\begin{algorithm}
\renewcommand{\thealgorithm}{A.4}
\caption{The SGIPSM for solving Problem 2 under Assumption $2$}
\label{algorithm:4}
\begin{algorithmic}[1]
\State Input: Real numbers $\alpha_1, \alpha_2, \gamma, \delta$; a positive real number $b$; real-valued functions $p, g$; $\bm{x}_{b,n}^{(\alpha )} \in \mathbb{S}_{b,n}^{(\alpha )}$, for some positive integer $n$; ${\mathbf{Q}}_b^{(1)}; {\mathbf{Q}}_b^{(2)}$; solution points vector $\bm{z} = (z_i) \in \mathbb{R}^m: z_i \in (0, b], i = 0, \ldots, m$, for some $m \in \mathbb{Z}^{+}_0$.
\State Output: Approximate solution vector $y_{\text{app},n}\left(\bm{z}\right)$.
\State Calculate $\mathbf{H}_{b,n}^{(\alpha)}$ and ${\bm{\hat x}}_{b,n}^{(\alpha )}$ using Eqs. \eqref{eq:53} and \eqref{hatx1}, respectively.
\State Calculate the coefficient matrix ${\mathcal{C}}$ and the constant vector ${\mathcal{D}}$ using Eqs. \eqref{eq:C1}-\eqref{eq:g1}.
\State Solve the linear algebraic system \eqref{eq:58newC1} for $\bm{\Psi} = \left[\bm{\Phi};y_{\text{app},n}(0)\right]$.
\State Calculate the approximate solution vector ${y_{\text{app},n}}\left( {{\bm{x}}_{b,n}^{(\alpha )}} \right)$ using Eq. \eqref{eq:14new}.
\State Calculate ${y_{\text{app},n}}\left( \bm{z} \right)$ using Eq. \eqref{eq:57}; Stop.
\end{algorithmic}
\end{algorithm}


\end{document}